\documentclass[11pt,reqno]{amsart}
\usepackage[skip=5pt plus1pt, indent=20pt]{parskip}

\usepackage{palatino}
\usepackage{mathpazo}
\usepackage{amsmath,amssymb,amsxtra,color,calligra,mathrsfs,tcolorbox}

\usepackage{xcolor} % Required for specifying custom colours
\colorlet{mdtRed}{red!50!black}
\colorlet{dblue}{blue!50!black}
\usepackage[colorlinks,pagebackref=true]{hyperref}
\hypersetup{linkcolor=dblue,citecolor=dblue,filecolor=dullmagenta,urlcolor=mdtRed}
\renewcommand*{\backref}[1]{}
\renewcommand*{\backrefalt}[4]{[{%
		\ifcase #1 Not cited.%
		\or $\uparrow$~#2.%
		\else $\uparrow$~#2.%
		\fi%
	}]}
\usepackage[all]{xy}
\usepackage{tikz,tikz-cd,tkz-graph,enumerate}
\usetikzlibrary{matrix,arrows,decorations.pathmorphing}

\usepackage[us,12hr]{datetime}
\usepackage{comment}

\DeclareMathOperator{\Hom}{{\rm Hom}}

\DeclareMathOperator{\Vect}{{\bf Vect}}
\DeclareMathOperator{\ParVect}{{\bf PVect}}

\newcommand{\mf}[1]{\mathfrak{#1}}
\newcommand{\mc}[1]{\mathcal{#1}}
\newcommand{\ms}[1]{\mathscr{#1}}
\newcommand{\bb}[1]{\mathbb{#1}}

\newcommand{\dual}{^\vee}

\newcommand{\diagram}[1]{\begin{align}
		\xymatrix{#1}
\end{align}}

\newcommand*{\homsheaf}{\mathcal{H}\,\,\kern -5pt {\large om}}
\newcommand{\parhomsheaf}{\mathcal{PH}\,\,\kern -5pt {\large om}}

\numberwithin{equation}{subsection}

\newtheorem{theorem}[equation]{Theorem}

\newtheorem{lemma}[equation]{Lemma}
\newtheorem{proposition}[equation]{Proposition}

\newtheorem{claim}[equation]{Claim}

\theoremstyle{definition}
\newtheorem{definition}[equation]{Definition}
\newtheorem{remark}[equation]{Remark}

\makeatletter
\newcommand\fnsymb[1]{\textsuperscript{\@fnsymbol{#1}}}
\newcommand\fnletter[1]{\lowercase{\textsuperscript{\@alph{#1}}}}
\newcommand\fnnum[1]{\textsuperscript{#1}}
\makeatother

\usepackage{marvosym} % For Email/Letter icons. 
\renewcommand{\email}[2][1]{\thanks{\textit{Email address}#1: \href{mailto:#2}{#2}}}

\renewcommand{\address}[2][1]{\thanks{\textit{Address}#1: #2}} 

\begin{document}
	
	\baselineskip=15.5pt 
	
	\title[Orthogonal and Symplectic Parabolic Connections and Root Stacks]{Orthogonal and Symplectic Parabolic Connections and Stack of Roots}
	
	\author[S. Chakraborty]{Sujoy Chakraborty\fnnum{1}}
	
	\address[\fnnum{1}]{Department of Mathematics, 
		Indian Institute of Science Education and Research Tirupati, Andhra Pradesh - 517507, India.}
	
	\email[\fnnum{1}]{sujoy.cmi@gmail.com}
	
	\author[S. Majumder]{Souradeep Majumder\fnnum{2}}
	
	\address[\fnnum{2}]{Department of Mathematics, 
		Indian Institute of Science Education and Research Tirupati, Andhra Pradesh - 517507, India.} 
	
	\email[\fnnum{2}]{souradeep@iisertirupati.ac.in}
	
	%]	
	%	\thanks{Corresponding author: Sujoy Chakraborty}
	
	
	
	\begin{abstract}
		Let $D$ be an effective divisor on a smooth projective variety $X$ over an algebraically closed field $k$ of characteristic $0$. We show that there is a one-to-one correspondence between the class of orthogonal (respectively, symplectic) parabolic vector bundles on $X$ with parabolic structure along $D$ and having rational weights and the class of orthogonal (respectively, symplectic) vector bundles on certain root stacks associated to this data. Using this, we describe the orthogonal (respectively, symplectic) vector bundles on the root stack as reductions of the structure group to orthogonal (respectively, symplectic) groups. When $D$ is a divisor with strict normal crossings, we prove a one-to-one correspondence between the class of orthogonal (respectively, symplectic) parabolic connections on $X$ with rational weights, and the class of orthogonal (respectively, symplectic) logarithmic connections on certain fiber product of root stacks with poles along a divisor with strict normal crossings.
	\end{abstract}
	
	%	\baselineskip=15.5pt 
	%	
	%	\title[Orthogonal and symplectic parabolic bundles and stack of roots]{Orthogonal and symplectic parabolic bundles and root stacks}
	%	
	%	\subjclass[2010]{14D23, 14H60, 53B15, 53C05}
	%	\keywords{Parabolic bundle; Root stack} 	
	
	\maketitle 
	
	\section{Introduction}
	Parabolic vector bundles were introduced by Mehta and Seshadri \cite{MehSes80} 
	on a compact Riemann surface $X$ together with a set $D$ of distinct marked points. 
	They generalized the celebrated Narasimhan-Seshadri theorem in the case of Riemann surfaces with punctures. Later, Simpson in \cite{Sim90} reformulated and extended the definition of parabolic bundles to the case of schemes $X$ of higher dimension with a normal crossings divisor $D$. When the weights are all rational, and therefore can be assumed to be in $\frac{1}{r}\bb Z$ for some integer $r$, it amounts to a locally free sheaf $\mc E$ of finite rank, together with a filtration
	\begin{align*}
		\mc E = \mc E_0\supset \mc E_{\frac{1}{r}}\supset \mc E_{\frac{2}{r}} 
		\supset\cdots\supset \mc E_{\frac{r-1}{r}}\supset \mc E_1= \mc E(-D)
	\end{align*}
	Parabolic bundles, parabolic connections and their moduli spaces have since been an active area of research, and have found themselves used in a number of important applications. More recently, parabolic connections have made appearance in the work of Donagi and Pantev on Geometric Langlands conjecture using Simpson's non-abelian Hodge theory \cite{DonPan19}. 
	
	When the weights are rational, it is known that parabolic bundles can be 
	interpreted equivalently as equivariant bundles on a suitable ramified 
	Galois cover \cite{Bis97}, and later, a more intrinsic interpretation was 
	found using vector bundles on certain algebraic stacks associated to $(X,D)$, 
	namely the  stack of roots \cite{Bor07}. More precisely, there is a Fourier-like 
	correspondence between parabolic vector bundles on a scheme and ordinary vector 
	bundles on certain stack of roots. This naturally raised the question of 
	understanding the parabolic connections on smooth varieties through such 
	Fourier-like correspondence over root stacks. This has been shown to be 
	true over curves \cite{BisMajWon12,LorSaiSim13}, and very recently over higher 
	dimensions \cite{BorLar23}. 
	
	Let $X$ be a smooth projective variety over an algebraically closed field $k$ together with an effective Cartier divisor $D$. Let $r$ be a positive integer. To this data, one can associate a stack $\mf{X} := \mf{X}_{(\mc{O}(D),s_D,r)}$ of $r$-th roots for the line bundle $\mc{O}_X(D)$ and its canonical section $s_D$ vanishing along $D$. It was shown by N. Borne \cite{Bor07} that there is an equivalence between the category of algebraic  vector bundles on $\mf{X}$ and the category of parabolic vector bundles on $X$ with parabolic structures along $D$ and having rational weights as multiples of $\frac{1}{r}$. 
	Our first main result is an extension of this equivalence to the case of orthogonal and symplectic parabolic vector bundles. More precisely, fix a line bundle $\mc{L}$ on $\mf{X}$, and let $L_\bullet$ be the corresponding parabolic line bundle on $X$. We prove the following:
	\begin{theorem}[\text{Theorem \ref{thm:borne correspondence for orthogonal and symplectic bundles}}]
		There is an equivalence between the category of orthogonal (respectively, symplectic) parabolic vector bundles on $X$ taking values in $L_\bullet$ (Defintiion \ref{def:orthogonal and symplectic parabolic bundles}) and with rational weights as multiples of $\frac{1}{r}$, with the category of orthogonal (respectively, symplectic) vector bundles on $\mf{X}$ taking values in $\mc{L}$ (Definition \ref{def:orthogonal and symplectic bundles on stacks}).
	\end{theorem}
	As an application, we show the following:
	\begin{proposition}[\text{Proposition \ref{prop:orthogonal and symplectic bundles as reductions of structure group}}]
		Let $X, D$ and $r$ be as above. The class of orthogonal (respectively, symplectic) bundles of rank $n$ on the stack $\mf{X}$ of $r$-th roots such that the bilinear form takes values in $\mc{O}_{\mf{X}}$ (cf. Definition \ref{def:orthogonal and symplectic bundles on stacks}) are in bijection with the class of principal $O(n)$-bundles (respectively, principal $\textnormal{Sp}(n)$-bundles) on $\mf{X}$, where in the symplectic case, $n$ is assumed to be even.
	\end{proposition}
	In the final section, we assume $D$ to be a divisor with strict normal crossings, meaning its irreducible components are smooth and they intersect each other transversally.  If $\textbf{D} :=(D_i)_{1\leq i \leq s}$ are the of components of $D$, we fix positive integers $(r_1,r_2,\cdots,r_s)$, and consider the product stack $\mf{X}_{(\bf D,r)}:= \prod_{i=1}^{s}\mf{X}_{(\mc{O}(D_i),s_{D_i},r_i)}$. The data of $D$ gives rise to a divisor $\mf{D}$ on $\mf{X}$ with strict normal crossings. Inspired by \cite{BorLar23}, where the authors establish an equivalence between the category of parabolic connections on $X$ and the category of logarithmic connections on certain root stacks, we extend their result to orthogonal and symplectic parabolic connections. 
	\begin{theorem}[\text{Theorem \ref{thm:borne correspondence for orthogonal and symplectic connections}}]
		Fix a line bundle $\mc{L}$ together with a logarithmic connection $\nabla_{\mc{L}}$ on $(\mf{X}_{(\bf D,r)},\mf{D})$. Let $(L_{\bullet},\nabla_{L_\bullet})$ be the corresponding parabolic connection on $(X,\bf{D})$ under Borne's equivalence. There is a one-to-one correspondence between the class of orthogonal (respectively, symplectic) vector bundles on $\mf{X}_{(\bf D,r)}$ taking values in $\mc{L}$ equipped with a compatible logarithmic connection (in the sense of \eqref{diagram:orthogonal and symplectic stacky logarithmic connections diagram}), and the class of orthogonal (respectively, symplectic) parabolic vector bundles on $X$ taking values in $L_\bullet$ equipped with a compatible parabolic connection (in the sense of Definition \ref{def:orthogonal and symplectic parabolic connection}).\\
		Furthermore, this equivalence restricts to a correspondence between the class of algebraic connections on $(\mf{X}_{(\bf D,r)},\mf{D})$ compatible with an orthogonal (respectively, symplectic) form and the class of compatible strongly parabolic connections on $(X,\bf D)$ compatible with an orthogonal (respectively, symplectic) form.
	\end{theorem}

	\section{Orthogonal and symplectic vector bundles on Deligne-Mumford stacks}
	We would like to extend the notion of symmetric and anti-symmetric bilinear forms on parabolic vector bundles on schemes as well as on vector bundles defined on stacks. This can be achieved using what is known as the \textit{transpose} of a morphism between such bundles. Let us first describe it for vector spaces over a field. Let $V$ and $L$ be two vector spaces over a field $K$ of dimensions $n$ and $1$ respectively. Let 
	$T: V \rightarrow \Hom(V,L)$
	be a $K$-linear map. The \textit{transpose} of $T$ is the unique $K$-linear map 
	$T^t : V\rightarrow \Hom(V,L)$
	satisfying
	\begin{align}\label{eqn:defining propetry of transpose}
		(T(v))(v')= ((T^t)(v'))(v)\,\,\forall\,v,v'\in V\,.
	\end{align}
	It is not hard to see that $T^t$ coincides with the following composition of linear maps:
	\begin{align}\label{eqn:transpose alternate description}
		V\simeq V\otimes K \stackrel{\simeq}{\longrightarrow}V\otimes (L^{\dual}\otimes L) \stackrel{\simeq}{\longrightarrow} \text{Hom}(V,L)^{\dual}\otimes L \stackrel{T^{\dual}\otimes \text{Id}_L}{\DOTSB\relbar\joinrel\relbar\joinrel\relbar\joinrel\rightarrow} V^{\dual}\otimes L \simeq \text{Hom}(V,L)
	\end{align}
	where $T^{\dual}: \text{Hom}(V,L)^{\dual}\rightarrow V^{\dual}$ is the dual map of $T$. \\
	%If we fix a basis for $V$ and $W$, then the matrix of $T^t$ is exactly represented by the transpose matrix $A^t$ with respect to these bases, where $A$ is the matrix of $T$ in these bases. 
	Next, with $V$ and $L$ as above, start with a $K$-bilinear map $\phi : V\otimes V\rightarrow L$. Via the Hom-tensor adjunction, consider the map 
	\begin{align}
		\widetilde{\phi} : V &\rightarrow \text{Hom}(V,L)\\
		v &\mapsto \phi(v,\rule{.2cm}{0.15 mm})
	\end{align} 
	then $(\widetilde{\phi})^t$  in \eqref{eqn:defining propetry of transpose} is given by sending $v \mapsto \phi(\ \rule{.2cm}{0.15 mm}\ ,v)\ \ \forall \,v\in V$.
	Thus, checking if the bilinear form $\phi$ is symmetric or anti-symmetric in the usual sense is same as checking whether $(\widetilde{\phi})^t = \widetilde{\phi}$ or $(\widetilde{\phi})^t = -\widetilde{\phi}$, respectively.\\
	As the maps involved in the description of $T^t$ in \eqref{eqn:transpose alternate description} are canonical, it follows that the notion of transpose makes sense if we replace vector spaces with vector bundles. More precisely, let $K$ be a field as before. Let $E$ be a vector bundle and $L$ be a line bundle on a $K$-scheme $X$. Let	$\alpha: E\rightarrow \homsheaf(E\ ,L)$  be a vector bundle morphism. Consider the composition \eqref{eqn:transpose alternate description} where  $T$ is replaced with $\alpha$ and $K$ is replaced with the trivial bundle $\mc{O}_X$. Namely, 
	\begin{align}\label{eqn:transpose alternate description for vector bundles}
		E \simeq E\otimes \mc{O}_X \stackrel{\simeq}{\longrightarrow} E\otimes (L^{\dual}\otimes L) \stackrel{\simeq}{\longrightarrow} \homsheaf(E\ ,L)^{\dual} \otimes L \stackrel{\alpha^{\dual}\otimes \text{Id}_L}{\DOTSB\relbar\joinrel\relbar\joinrel\relbar\joinrel\rightarrow} E^{\dual}\otimes L \simeq \homsheaf(E\ ,L)
	\end{align}
	We call this composition as the \textit{transpose} of $\alpha$:
	\begin{align}\label{eqn:transpose of a vector bundle morphism}
		\alpha^t: E\rightarrow \homsheaf(E\ ,L)\ .
	\end{align}
	It is clear that $\alpha^t$ satisfies the relation \eqref{eqn:defining propetry of transpose} when $v$ and $v'$ are replaced by local sections of $E$. 
	%(cf. \cite[Propositoion 1.2]{Hit07} for a different proof).
	%	Moreover, for a smooth algebraic stack $\mf{X}$ over $K$ and a vector bundle morphism  $\varphi:\mc{E}\rightarrow\homsheaf(\mc{E}\ ,\mc{L})$ on $\mf{X}$ ($\mc{L}$ is a line bundle), we can work over smooth charts for $\mf{X}$ to obtain a well-defined \textit{transpose} morphism in this case as well: 
	%	$$\varphi^t: \mc{E}\rightarrow \homsheaf(\mc{E}\ ,\mc{L})\ .$$
	The same result was proved in \cite[Proposition 1.2]{Hit07} as well, which we recall here for reference.
	\begin{lemma}\label{lem:transpose of vector bundle morphisms}
		Let $X$ be a scheme over $K$. Let $E\rightarrow X$ be a vector bundle of rank $n$ on $X$, and let $L\rightarrow X$ be a line bundle on $X$. Let 
		$\alpha : E \rightarrow \homsheaf(E,L)$ be a morphism of vector bundles. Then there is a well-defined \textnormal{transpose} of $\alpha$, which is also a morphism of vector bundles $\alpha^t:E\rightarrow \homsheaf(E,L)\,.$
	\end{lemma}
	
	\begin{proof}
		Follows from the discussion above, or alternatively from \cite[Proposition 1.2]{Hit07}\,. Here we mention that, although the statement is proved in [\textit{loc. cit.}] assuming $\dim \,X=1$, the same proof works for higher dimensions as well.
	\end{proof}
	If the vector bundle $E$ is equipped with a bilinear form with values in a line bundle $L$, i.e. an $\mc{O}_X$-linear morphism
	\begin{align}
		\phi : E \otimes E \longrightarrow L
	\end{align}
	the Hom-tensor adjunction gives rise to a morphism $\widetilde{\phi} : E\longrightarrow \homsheaf(E\ ,L)$ in the following manner. The identity morphism of $E$ gives rise to a map
	$$\iota:\mc{O}_X\longrightarrow \homsheaf(E,E) \simeq E^{\dual}\otimes E$$
	Next, consider the composition
	\begin{align}
		E\simeq E\otimes \mc{O}_X \stackrel{\text{Id}_E\otimes\iota}{\DOTSB\relbar\joinrel\relbar\joinrel\rightarrow} E\otimes(E^{\dual}\otimes E) \simeq E^{\dual}\otimes (E\otimes E) \stackrel{\text{Id}_{E^{\dual}}\otimes \phi}{\DOTSB\relbar\joinrel\relbar\joinrel\relbar\joinrel\rightarrow} E^{\dual}\otimes L \simeq \homsheaf (E\ ,L). 
	\end{align} 
	We call this resulting composition as $\widetilde{\phi}$:
	\begin{align}
		\widetilde{\phi} : E \longrightarrow \homsheaf(E\ ,L)
	\end{align}
	By \eqref{eqn:transpose of a vector bundle morphism}, we can consider its transpose $(\widetilde{\phi})^t$, and call the bilinear form $\phi$ to be \textit{symmetric} or \textit{anti-symmetric} depending on whether $$\widetilde{\phi}^t = \widetilde{\phi}\ \  \text{or}\ \  \widetilde{\phi}^t = -\widetilde{\phi}\quad\eqref{eqn:transpose of a vector bundle morphism}\ ,$$
	respectively.
	
	\subsection{Coherent sheaves on algebraic stacks}\label{subsection:coherent sheaves on stacks}
	Fix an algebraically closed field $k$ of characteristic zero. Let $\mf{X}$ be a smooth algebraic stack of finite type over the category of $k$-schemes. For any $k$-scheme $U$, let $\mf{X}(U)$ denote the fiber category over $U$. By the $2$-categorical version of the Yoneda lemma, the objects of $\mf{X}(U)$ are in bijection with morphisms $U\rightarrow \mf{X}$\,, and we shall often use this fact without any further comments. A coherent sheaf $\mc{F}$ on $\mf{X}$ consists of the data (cf. \cite[Definition 7.18]{Vis89}): for each smooth atlas $u:U\rightarrow \mf{X}$ (i.e. $U$ is a $K$-scheme and $u$ is a smooth morphism), we are given a coherent $\mc{O}_U$-module $\mc{F}_u$ on $U$, and for any  $2$-commutative diagram
	\begin{align}\label{diagram:smooth atlas}
		\xymatrix{&&U\ar[d]^{u} \\
			V \ar[rru]^{k} \ar[rr]_v && \mf{X}}
	\end{align}	
	together with a natural isomorphism $\phi: u\circ k \rightarrow v$, where $u$ and $v$ smooth atlases, there exists an isomorphism 
	\begin{align*}
		f_{\phi,k} : \mc{F}_v \simeq k^*\mc{F}_u
	\end{align*}
	such that for any diagram 
	\begin{align}
		\xymatrix{ W \ar[r]^{l} \ar[rd]_(0.4){w} & V \ar[r]^{k} \ar[d]^{v} & U \ar[ld]^(0.4){u} \\
			& \mf{X} &
		}
	\end{align}
	with both the triangles $2$-commutative, thus having natural isomorphisms $\phi: u\circ k\rightarrow v$ and $\psi:v\circ l \rightarrow w$, the following diagram commutes:
	\begin{align}
		\xymatrix{
			\mc{F}_{w} \ar[rr]^(0.4){f_{(\psi\circ\phi),(k\circ l)}} \ar[d]_{f_{\psi,l}} && (k\circ l)^*(\mc{F}_{u}) \ar[d]^{\cong}\\
			l^*{\mc{F}_{v}} \ar[rr]^{l^*(f_{\phi,k})} && l^*(k^*\mc{F}_{u})}
	\end{align}
	A coherent sheaf $\mc{E}$ on $\mf{X}$ is called a \textit{vector bundle} if each $\mc{E}_u$ is a locally free sheaf of finite rank for each atlas $U$.
	\subsection{Transpose of a morphism of vector bundles on algebraic stacks}
	We would like to extend the notion of transpose of vector bundle morphisms as above for algebraic stacks. 
	Given two vector bundles $\mc{E},\mc{F}$ on $\mf{X}$, we can construct their internal Hom-sheaf, denoted by $\homsheaf(\mc{E},\mc{F})$, which is again a vector bundle on $\mf{X}$. Let us describe the Hom-sheaf in detail. \\
	Consider two smooth atlases  $U\xrightarrow{u} \mf{X}$ and $V\xrightarrow{v}\mf{X}$ (i.e. $U,V$ are $k$-schemes and $u,v$ are smooth maps). Consider a $2$-commutative diagram as in \eqref{diagram:smooth atlas}, together with a natural isomorphism $\phi: u\circ k\rightarrow v$. Then $\mc{E}$ and $\mc{F}$ give rise to vector bundles  $\mc{E}_u$ and $\mc{F}_u$ on $U$ (respectively, $\mc{E}_v$ and $\mc{F}_v$ on $V$), together with isomorphisms
	\begin{align}\label{eqn:compatibility maps for stacky vector bundles}
		f_{\phi,k} : \mc{E}_v \stackrel{\simeq}{\longrightarrow} k^*\mc{E}_u\,\,\text{and}\,\,g_{\phi,k} : \mc{F}_v \xrightarrow{\simeq} k^*\mc{F}_u\,.
	\end{align}
	This induces an isomorphism 
	\begin{align}\label{eqn:compatibility maps on hom}
		\theta_{\phi,k}: \homsheaf(\mc{E}_v,\mc{L}_v) &\stackrel{\simeq}{\longrightarrow} k^*\homsheaf(\mc{E}_u,\mc{F}_u) \simeq \homsheaf(k^*\mc{E}_u,k^*\mc{F}_u)\\
		\text{given by}\,\ \ \ \ 
		\phi & \mapsto g_{\phi,k}\circ \phi\circ (f_{\phi,k})^{-1}\,.
	\end{align}
	The data of $\{\homsheaf(\mc{E}_u,\mc{F}_u)\mid u:U\rightarrow \mf{X} \,\ \text{is a smooth atlas}\}$, together with the maps $\theta_{\phi,k}$ for each $2$-commutative diagram \eqref{diagram:smooth atlas} constitute the sheaf $\homsheaf(\mc{E},\mc{F})$ on $\mf{X}$.
	
	\begin{proposition}
		Let $\mc{E}$ be a vector bundle on $\mf{X}$, and $\mc{L}$ be a line bundle on $\mf{X}$. Consider a morphism of vector bundles $\alpha: \mc{E}\rightarrow \homsheaf(\mc{E},\mc{L})$ on the stack $\mf{X}$\,. There is a well-defined \textnormal{transpose} of $\alpha$, denoted by $\alpha^t : \mc{E}\rightarrow \homsheaf(\mc{E},\mc{L})$\,.
	\end{proposition}
	\begin{proof}
		Consider two smooth atlases  $U\xrightarrow{u} \mf{X}$ and $V\xrightarrow{v}\mf{X}$ together with a $2$-commutative diagram as in \eqref{diagram:smooth atlas}, and a natural isomorphism $\phi: u\circ k\rightarrow v$. Then $\mc{E}$ and $\mc{L}$ give rise to vector bundles  $\mc{E}_u$ and $\mc{L}_u$ on $U$  (respectively, $\mc{E}_v$ and $\mc{L}_v$ on $V$), together with isomorphisms 
		\begin{align}
			f_{\phi,k} : \mc{E}_v \stackrel{\simeq}{\longrightarrow} k^*\mc{E}_u\,\,\text{and}\,\,g_{\phi,k} : \mc{F}_v \xrightarrow{\simeq} k^*\mc{F}_u\,.
		\end{align}
		Recall the Hom-sheaf $\homsheaf(\mc{E},\mc{L})$ constructed above. The morphism $\alpha : \mc{E}\rightarrow \homsheaf(\mc{E},\mc{L})$ gives rise to vector bundle morphisms 
		$$\alpha_u : \mc{E}_u\rightarrow \homsheaf(\mc{E}_u,\mc{L}_u)\ \ \text{and}\ \  \alpha_v: \mc{E}_v\rightarrow \homsheaf(\mc{E}_v,\mc{L}_v)$$
		such that the following diagram commutes:
		\begin{align}\label{diagram:morphism of vector bundles on stacks compatibility on atlas}
			\xymatrix{\mc{E}_v \ar[rr]^(.4){\alpha_v} \ar[d]_{f_{\phi,k}} & & \homsheaf(\mc{E}_v,\mc{L}_v) \ar[d]^{\theta_{\phi,k}}\\
				k^*\mc{E}_u \ar[rr]^(.4){k^*\alpha_u} && \homsheaf(k^*\mc{E}_u,k^*\mc{L}_u)
			}
		\end{align}
		where $\theta_{\phi,k}$ is as in \eqref{eqn:compatibility maps on hom}. Consider their transpose maps $\alpha_u^t$ and $\alpha_v^t$ [Lemma \ref{lem:transpose of vector bundle morphisms}]. It is clear that the construction of transpose commutes with pull-back, namely $(k^*\alpha_u)^t = k^*(\alpha_u^t)$\,.\\
		Now, in order to get a well-defined transpose map $\alpha^t$, we need to check the commutativity of the diagram 
		\begin{align}\label{diagram:transpose}
			\xymatrix{\mc{E}_v \ar[rr]^(.4){\alpha^t_v} \ar[d]_{f_{\phi,k}} & & \homsheaf(\mc{E}_v,\mc{L}_v) \ar[d]^{g_{\phi,k}}\\
				k^*\mc{E}_u \ar[rr]^(.4){k^*(\alpha^t_u)} && \homsheaf(k^*\mc{E}_u,k^*\mc{L}_u)
			}
		\end{align}
		It is enough to check the commutativity of this diagram fiber-wise. Thus, using  \eqref{eqn:compatibility maps on hom} and the diagram \eqref{diagram:morphism of vector bundles on stacks compatibility on atlas} we reduce to the following question: let $V, W$ be two vector spaces, $L, M$ be two $1$-dimensional vector spaces, and $f : V\xrightarrow{\simeq} W$, $g: L\xrightarrow{\simeq} M$ be two linear isomorphisms. Consider the induced isomorphism $\theta: \Hom(V,L) \rightarrow \Hom(W,M)$ by sending \ $\phi\mapsto g\circ \phi\circ f^{-1}$\,.
		Thus, we need to check the following.
		\begin{claim}
			If $\gamma: V\rightarrow \Hom(V,L)$ and $\delta: W\rightarrow \Hom(W,M)$ are two linear maps such that the diagram
			\diagram{V \ar[rr]^(.4){\gamma} \ar[d]_{f}^{\simeq} && \Hom(V,L) \ar[d]^{\theta}_{\simeq} \\
				W \ar[rr]^(.4){\delta} && \Hom(W,M)}
			commutes, i.e. $\theta\circ \gamma = \delta\circ f$, then we have $\theta\circ (\gamma^t) = (\delta^t)\circ f$\,.
		\end{claim}
		\textit{Proof of claim.}\ \ \ 
		Let $v\in V, w\in W$\,. We have
		\begin{align*}
			\quad\qquad\qquad\qquad\quad\left[(\theta\circ\gamma^t)(v)\right](w) = \left[\theta\left((\gamma^t)(v)\right)\right](w) &= \left[g\circ((\gamma^t)(v))\circ f^{-1}\right](w)\\
			&=g\left(\gamma^t(v)(f^{-1}(w))\right)\\
			&=g\left(\gamma(f^{-1}(w))(v)\right)\ [\text{by definition}]\\
			&=g\left((\gamma\circ f^{-1})(w)(v)\right)\\
			&=g\left((\theta^{-1}\circ\delta)(w)(v)\right)\\
			&=g\left((\theta^{-1}(\delta(w)))(v)\right)\\
			&=g\left((g^{-1}\circ\delta(w)\circ f)(v)\right)\\
			&=\delta(w)(f(v))\\
			&=\delta^t(f(v))(w)\\
			&=(\delta^t\circ f)(v)(w)
		\end{align*}
		\qquad\qquad \qquad\quad We have thus shown that $\theta\circ(\gamma^t)=(\delta^t)\circ f\,.$\hfill(\text{Claim proved}) {\parfillskip0pt\par}
		Thus the diagram \eqref{diagram:transpose} commutes, making $\alpha^t_u$ compatible with $\alpha^t_v$\,. This proves the existence of a well-defined morphism $\alpha^t : \mc{E}\rightarrow \homsheaf(\mc{E},\mc{L})$ as required\,.
	\end{proof}
	%This enables us to speak about symmetric and anti-symmetric morphisms $\mc{E}\rightarrow \homsheaf(\mc{E},\mc{L})$\,.
	%	\begin{definition}\label{def:symmetric and anti-symmetric morphisms for stacks}
		%		A morphism of vector bundles $\alpha: \mc{E}\rightarrow \homsheaf(\mc{E},\mc{L})$ is said to be \textit{symmetric} if $\alpha^t = \alpha$, and \textit{anti-symmetric} if $\alpha^t = -\alpha$\,.
		%	\end{definition}
	It can be checked that the transpose of the bundle morphism $\alpha: \mc{E}\rightarrow \homsheaf(\mc{E}\ ,\mc{L})$ above can be defined analogous to what was done for vector bundles on schemes in \eqref{eqn:transpose alternate description for vector bundles} by considering the composition
	\begin{align}\label{eqn:transpose alternate definition for stacky vector bundles}
		\mc{E}\simeq \mc{E}\otimes\mc{O}_{\mf{X}}\stackrel{\simeq}{\longrightarrow} \mc{E}\otimes(\mc{L}^{\dual}\otimes\mc{L}) \stackrel{\simeq}{\longrightarrow} \homsheaf(\mc{E}\ ,\mc{L})^{\dual}\otimes\mc{L}\stackrel{\alpha^{\dual}\otimes \text{Id}_{\mc{L}}}{\DOTSB\relbar\joinrel\relbar\joinrel\relbar\joinrel\rightarrow} \mc{E}^{\dual}\otimes\mc{L}\simeq \homsheaf(\mc{E}\ ,\mc{L})
	\end{align}
	and calling the resulting composition as the transpose of $\alpha$:
	\begin{align}
		\alpha^t:\mc{E}\longrightarrow\homsheaf(\mc{E}\ ,\mc{L})\ .
	\end{align}
	\subsection{Orthogonal and symplectic bundles on Deligne-Mumford stacks}\label{subsection:orthogonal and symplectic for stacks}
	Let $\mf{X}$ be a Deligne-Mumford stack of finite type over $k$ as before. Fix a line bundle $\mc{L}$ on $\mf{X}$. Let $\mc{E}$ a vector bundle of rank $n$. Let 
	\begin{align}
		\varphi: \mc{E}\otimes\mc{E}\longrightarrow \mc{L}
	\end{align}
	be a homomorphism of vector bundles on $\mf{X}$. The identity endomorphism of $\mc{E}$ produces for us a vector bundle morphism 
	\begin{align}
		j : \mc{O}_{\mf{X}} \longrightarrow \homsheaf(\mc{E},\mc{E})\simeq \mc{E}^{\dual}\otimes \mc{E}
	\end{align}
	Let us consider the following composition:
	\begin{align}\label{eqn:stacky tilde composition}
		\mc{E} \simeq \mc{E}\otimes \mc{O}_{\mf{X}} \stackrel{\text{Id}\otimes j}{\DOTSB\relbar\joinrel\relbar\joinrel\rightarrow} \mc{E}\otimes (\mc{E}^{\dual}\otimes \mc{E}) \simeq \mc{E}^{\dual}\otimes\left(\mc{E}\otimes \mc{E}\right) \stackrel{\text{Id}_{\mc{E}^{\vee}}\otimes\varphi}{\DOTSB\relbar\joinrel\relbar\joinrel\relbar\joinrel\rightarrow} \mc{E}^{\dual}\otimes\mc{L}\simeq \homsheaf(\mc{E},\mc{L})
	\end{align}
	Let  $\widetilde{\varphi}$ denote the composition: 
	\begin{align}\label{eqn:tilde of a stacky bilinear form}
		\widetilde{\varphi}:\mc{E}\longrightarrow \homsheaf(\mc{E},\mc{L})
	\end{align}
	\vspace{1cm}
	\begin{definition}\label{def:orthogonal and symplectic bundles on stacks}	
		Let $\mc{L}$ be a fixed line bundle on $\mf{X}$. \hfill
		\begin{enumerate}
			\item An \textit{orthogonal bundle on} $\mf{X}$ \textit{taking values in} $\mc{L}$ is a pair $(\mc{E},\varphi)$, where $\mc{E}$ is a vector bundle on $\mf{X}$ together with a morphism $\varphi: \mc{E}\otimes\mc{E}\longrightarrow \mc{L}$ of vector bundles, such that the map $\widetilde{\varphi}$ is \textit{symmetric} (i.e. $\widetilde{\varphi}^t = \widetilde{\varphi}$) and an \textit{isomorphism}.\\
			\item A \textit{symplectic bundle on} $\mf{X}$ \textit{taking values in} $\mc{L}$ is a pair $(\mc{E},\varphi)$, where $\mc{E}$ is a vector bundle on $\mf{X}$ together with a morphism $\varphi: \mc{E}\otimes\mc{E}\longrightarrow \mc{L}$ of vector bundles, such that the map $\widetilde{\varphi}$ is \textit{anti-symmetric} (i.e. $\widetilde{\varphi}^t = -\widetilde{\varphi}$) and an \textit{isomorphism}.
		\end{enumerate}
	\end{definition}
	\subsection{Root stacks}\label{subsection:root stacks}
	A \textit{generalized Cartier divisor} on $X$ is a tuple $(L, s)$, 
	where $L$ is a line bundle on $X$ and $s \in \Gamma(X,L)$. 
	%\end{definition}
	Given a generalized Cartier divisor $(L, s)$ and a positive integer $r$, 
	let $\mf X_{(L,\, s,\, r)}$ be the stack of $r$-th roots of $(L, s)$; 
	it is a Deligne-Mumford stack \cite{Cad07}. 
	As a category, its objects are given by tuples $\left(f : U \to X, M, \phi, t\right)$, 
	where 
	\begin{itemize}
		\item $f : U \to X$ is a morphism of $k$-schemes, 
		\item $M$ is a line bundle on $U$, 
		\item $t \in \Gamma(U, M)$, and 
		\item $\phi : M^{\otimes r} \stackrel{\simeq}{\longrightarrow} f^*L$ 
		is an $\mc O_U$-module isomorphism such that $\phi(t^{\otimes r}) = f^*s$. 
	\end{itemize}
	A morphism from 
	$\left(f : U \to X, M, \phi, t\right)$ to $\left(f' : U' \to X, M', \phi', t'\right)$ 
	in $\mf X_{(L,\, s,\, r)}$ is given by a pair $(g, \psi)$, where 
	\begin{itemize}
		\item $g : U \to U'$ is a morphism of $X$-schemes, and 
		\item $\psi : M \to g^* M'$ is an $\mc O_U$-module isomorphism such that $\psi(t)= g^*(t')$. 
	\end{itemize}
	There is a natural morphism of stacks $\pi : \mf{X}_{(L,\, s,\, r)} \to X$ that sends 
	the object $\left(f : U \to X, M, \phi, t\right)$ of $\mf X_{(L,\, s,\,r)}$ 
	to the $X$-scheme $f : U \to X$, and a morphism $(g, \psi)$, as above, 
	to the morphism of $X$-schemes $g : U \to U'$. We shall be interested in the case when we have an effective Cartier divisor $D$ on $X$. Consider the corresponding line bundle ${\mc O}(D)$, and let $s_D$ denote the tautological section of ${\mc O}(D)$ vanishing along $D$. For a given positive integer $r$, we shall consider the $r$-th root stack ${\mf X}_{({\mc O}(D),s_D,r)}$ together with the canonical morphism $\pi:\mf{X}_{(\mc{O}(D),s_D,r)}\longrightarrow X$\ .  It is well-known that $\mf{X}_{(\mc{O}(D),s_D,r)}$ is a Deligne-Mumford stack of finite type over $X$ \cite{Cad07}.
	\section{Orthogonal and symplectic parabolic bundles}\label{section:orthogonal and symplectic parabolic bundles}
	The notion of parabolic vector bundles over connected smooth projective curves were introduced by Mehta and Seshadri in \cite{MehSes80}. It was later extended to higher dimensional varieties by Simpson, Maruyama and Yokogawa \cite{MarYok92,Yok95,Sim90}. Let $X$ be a smooth projective variety over $k$\,.  Let $D$ be an effective Cartier divisor on $X$. Let us fix a positive integer $r$. We consider the set $\bb R$ as an indexing category, where we have a unique arrow $a\rightarrow b$ if and only if $a\leq b$. Let $\textnormal{\bf Vect}(X)$ denote the category of algebraic vector bundles on $X$. The following is taken from \cite{Yok95}.
	\begin{definition}\label{def:parabolic bundles}
		A \textit{parabolic vector bundle} on $X$ with parabolic structure along $D$ is a tuple $(E_{\bullet}, {\mf j}_{E_{\bullet}})$, where 
		\begin{align}
			E_{\bullet} : {\bb R}^{op} \longrightarrow \textnormal{\bf Vect}(X)
		\end{align}
		is a functor, and 
		\begin{align}
			{\mf j}_{E_{\bullet}} : E_{\bullet+1}  \stackrel{\simeq}{\longrightarrow} E_{\bullet}\otimes \mc{O}_X(-D)
		\end{align}
		is a natural isomorphism, making the following diagram commutes:
		\begin{align}
			\xymatrix{
				E_{\bullet+1} \ar[r] \ar[d]_{{\mf j}_{E_{\bullet}}} & E_\bullet \\
				E_\bullet \otimes \mc{O}_X(-D) \ar[ru]
			}
		\end{align}
		where the left vertical arrow is the canonical inclusion. Furthermore, for each $\alpha,\beta$ such that $\alpha\leq \beta<\alpha +1$, the quotient $\dfrac{E_{\alpha}}{E_{\beta}}$ is a locally free $\mc{O}_D$-module.
		We also require that there exists a sequence of real numbers 
		\begin{align}
			0\leq \alpha_1 < \alpha_2 < \cdots \alpha_m< 1
		\end{align}
		such that each $E_{\alpha_i} \rightarrow E_{\alpha}$ is an isomorphism for all  real numbers $\alpha_{i-1}<\alpha \leq \alpha_i$ (by convention, $\alpha_0 = 0, \alpha_{m+1}=1$). The collection $\{\alpha_i\mid 1\leq i\leq m\}$ are called \textit{weights} of $(E_{\bullet}, {\mf j}_{E_\bullet})$ if furthermore $E_{\alpha_{i+1}}\rightarrow E_{\alpha_i}$ are not isomorphisms $\forall \ 1\leq i\leq m$.
	\end{definition}
	We shall sometime suppress the isomorphism ${\mf j}_{\bullet}$, and denote a parabolic bundle as above simply by $E_{\bullet}$. Any algebraic vector bundle $E$ on $X$ can be seen as a parabolic vector bundle with parabolic structure along $D$. The resulting parabolic bundle is denoted by $(E)_{\bullet}$, and it is defined as  
	\begin{align}\label{def:special parabolic structure}
		(E)_t = E(-[t]D) \ \ \forall \ t\in {\bb R}\ .\qquad\text{\cite{Yok95}}
	\end{align}
	We say that the resulting parabolic bundle has \textit{special structure}.\\
	For two parabolic bundles $E_{\bullet}$ and $E_{\bullet}$, a natural transformation $\phi_{\bullet}: E_{\bullet}\rightarrow F_\bullet$ is said to be a \textit{parabolic morphism} if the following diagram commutes:
	\begin{align}
		\xymatrix{
			E_{\bullet+1} \ar[rr]^{\phi_{\bullet + 1}} \ar[d]_{{\mf j}_{E_\bullet}}^{\simeq} && F_{\bullet +1} \ar[d]^{{\mf j}_{F_\bullet}}_{\simeq}  \\
			E_{\bullet}\otimes \mc{O}_X(-D) \ar[rr]^{\phi_{\bullet}\otimes \text{Id}} &&  F_{\bullet}\otimes \mc{O}_X(-D)
		}
	\end{align} 
	For any two parabolic bundles $E_{\bullet}$ and $F_{\bullet}$\,, this category admits their tensor product $(E_{\bullet}\otimes F_{\bullet})_{\bullet}$ as well as their internal Hom object, namely the parabolic Hom-sheaf $\parhomsheaf(E_{\bullet}\ ,F_{\bullet})_{\bullet}$ \cite[\S 2.1]{Bis03,BalaBisNag01,Yok95}\ .  In particular, any parabolic bundle $E_{\bullet}$ admits a parabolic dual, denoted by $E^{\dual}_{\bullet}$:
	$$E_{\bullet}^{\dual} : = \parhomsheaf(E_{\bullet}\ ,({\mc O}_X)_{\bullet})_{\bullet} $$
	where $({\mc O}_X)_\bullet$ is the trivial line bundle $\mc{O}_X$ with special parabolic structure as in \eqref{def:special parabolic structure}. We refer the reader to \cite{Yok95} for further details.\\
	Unless specified otherwise, from here onward we shall consider parabolic bundles with parabolic structure along $D$.
	\subsection{Orthogonal and symplectic parabolic bundles}
	When $X$ is a curve, the notion of orthogonal and symplectic parabolic vector bundles over $X$ were introduced in \cite{BisMajWon11}\,. Here we extend their definition for higher dimensions.  Fix a parabolic line bundle $L_{\bullet}$ on $X$. Let $E_{\bullet}$ be a parabolic vector bundle of rank $n$ together with a parabolic morphism
	\begin{align}
		\phi_{\bullet} : (E_{\bullet} \otimes E_{\bullet})_\bullet \longrightarrow  L_{\bullet}
	\end{align} 
	The identity parabolic endomorphism of $E_{\bullet}$ gives rise to an $\mc{O}_X$-module homomorphism 
	\begin{align}\label{eqn:morphism coming from identity}
		\mc{O}_X \longrightarrow \parhomsheaf (E_{\bullet},E_{\bullet})
	\end{align}
	Now, let $(\mc{O}_X)_{\bullet}$  denotes the sheaf $\mc{O}_X$ with the special structure. Recall from \eqref{def:special parabolic structure} that it is defined as
	\begin{align}
		(\mc{O}_X)_{t} = \mc{O}_X([-t]D)\,\,\forall\,t\in \bb{R}\,.\qquad\text{\cite{Yok95}}
	\end{align}
	We claim that the morphism in \eqref{eqn:morphism coming from identity} gives rise to a parabolic morphism $$(\mc{O}_X)_{\bullet}\rightarrow \parhomsheaf(E_{\bullet}\ ,E_\bullet)_\bullet\,.$$
	To see this, first of all, any parabolic sheaf $F_{\bullet}$ satisfies the following:
	\begin{align}
		&F_{t} \simeq \parhomsheaf((\mc{O}_X)_{\bullet}\ , F_{\bullet})_{t}\ \ \ \,\forall\,{t\in \bb{R}}  \qquad \text{\cite[Example 3.3]{Yok95}}\\
		\implies&F_0\simeq \parhomsheaf((\mc{O}_X)_{\bullet}\ , F_{\bullet})_0 \\
		\implies& \homsheaf(\mc{O}_X,F_0)\simeq \parhomsheaf((\mc{O}_X)_{\bullet}\ , F_{\bullet}) \qquad\text{\cite[Definition 3.2]{Yok95}}
	\end{align}
	In particular, for the parabolic sheaf $F_\bullet = \parhomsheaf(E_{\bullet}\,,E_{\bullet})_{\bullet}$ we get 
	\begin{align}
		\homsheaf\left(\mc{O}_X, \parhomsheaf\left(E_{\bullet},E_{\bullet}\right)\right) &=  \homsheaf\left(\mc{O}_X, \parhomsheaf\left(E_{\bullet},E_{\bullet}\right)_0\right)\qquad\text{\cite[Definition 3.2]{Yok95}}\\
		&\simeq \parhomsheaf\left((\mc{O}_X)_{\bullet}\ , \parhomsheaf(E_{\bullet},E_{\bullet})_{\bullet}\right)\quad\text{\cite[Example 3.3]{Yok95}}
	\end{align}
	Thus the morphism \eqref{eqn:morphism coming from identity} gives rise to a morphism of parabolic bundles 
	\begin{align}\label{eqn:parabolic morphism coming from identity}
		\iota_{\bullet} : (\mc{O}_X)_{\bullet} \longrightarrow \parhomsheaf(E_{\bullet},E_{\bullet})_{\bullet} \simeq (E_{\bullet}^{\dual}\otimes E_{\bullet})_{\bullet}
	\end{align}
	where the last isomorphism follows from \cite[Lemma 3.6]{Yok95}.  Consider the composition
	\begin{align}\label{eqn:parabolic tilde composition}
		E_{\bullet} \simeq (E_{\bullet}\otimes (\mc{O}_X)_{\bullet})_{\bullet} \stackrel{\text{Id}_{E_{\bullet}}\otimes\iota_{\bullet}}{\DOTSB\relbar\joinrel\relbar\joinrel\relbar\joinrel\rightarrow}(E_{\bullet}\otimes (E^{\dual}_{\bullet}\otimes E_{\bullet})_{\bullet})_{\bullet}\simeq (E_{\bullet}^{\dual}\otimes (E_{\bullet}\otimes E_{\bullet})_{\bullet})_{\bullet} \stackrel{\text{Id}\otimes \phi_{\bullet}}{\DOTSB\relbar\joinrel\relbar\joinrel\relbar\joinrel\rightarrow} (E_{\bullet}^{\dual}\otimes L_{\bullet})_{\bullet}
	\end{align}
	Let $\widetilde{\phi_{\bullet}}$ denote the resulting composition:
	\begin{align}\label{eqn:tilde of a parabolic bilinear form}
		\widetilde{\phi_{\bullet}} : E_{\bullet} \longrightarrow \parhomsheaf(E_{\bullet}\,,L_{\bullet})_{\bullet}
	\end{align}
	Now, one can define the transpose of \ $\widetilde{\phi_{\bullet}}$ along the same lines as we did for usual vector bundles. More generally, assume that we have a parabolic morphism $$\psi_{\bullet} : E_\bullet \rightarrow \parhomsheaf(E_\bullet\ , L_\bullet)_{\bullet}\ .$$
	We have already seen in \eqref{eqn:parabolic morphism coming from identity} that the identity parabolic morphism of $L_\bullet$ gives rise to a parabolic morphism
	\begin{align}\label{eqn:identity parabolic morphism}
		i_{\bullet}: (\mc{O}_X)_{\bullet} \longrightarrow (L^{\dual}_{\bullet}\otimes L_{\bullet})_{\bullet}
	\end{align}
	%	As $L_\bullet$ is of rank $1$, this is in fact a parabolic isomorphism, since we also have the following inclusion (cf. \cite[Definition 1.3]{Bot10}):
	%	\begin{align}
		%		\parhomsheaf(L_{\bullet}\ ,L_\bullet)_t = \parhomsheaf(L_\bullet\ ,L_{\bullet+t}) \subset \homsheaf(L_0\ , L_t) = \homsheaf\left(L_0\ ,L_0([-t]D)\right) \simeq \mc{O}_{X}([-t]D)
		%	\end{align}
	%	which produces an inverse map $\parhomsheaf(L_\bullet\ ,L_\bullet)_{\bullet} \hookrightarrow (\mc{O}_X)_{\bullet}$ to $i_{\bullet}$. 
	Next, consider the composition 
	\begin{align}\label{eqn:transpose alternate description for parabolic bundles}
		E_\bullet \stackrel{\simeq}{\rightarrow}(E_{\bullet}\otimes (\mc{O}_X)_{\bullet})_{\bullet}\stackrel{\text{Id}\otimes i_{\bullet}}{\DOTSB\relbar\joinrel\relbar\joinrel\rightarrow} (E_{\bullet}\otimes (L^{\dual}_{\bullet}\otimes L_{\bullet})_{\bullet})_{\bullet} \stackrel{\simeq}{\rightarrow} (\parhomsheaf(E_{\bullet}\ ,L_{\bullet})^{\dual}_{\bullet} \otimes L_{\bullet})_{\bullet}\stackrel{\psi^{\dual}_{\bullet}\otimes \text{Id}}{\DOTSB\relbar\joinrel\relbar\joinrel\relbar\joinrel\rightarrow} (E^{\dual}_{\bullet}\otimes L_{\bullet})_{\bullet}
	\end{align}
	We denote this resulting composition as the \textit{transpose} of $\psi_{\bullet}$: 
	\begin{align}
		\psi_{\bullet}^{t}: E_{\bullet} \longrightarrow \parhomsheaf(E_\bullet\ ,L_\bullet)_{\bullet}
	\end{align}
	
	\begin{remark}
		The morphism $i_{\bullet}$ constructed in \eqref{eqn:identity parabolic morphism} is in fact a parabolic isomorphism. To see this, we observe that since $L_\alpha$ is of rank $1$, for any $\alpha\in [0,1)$ it follows from Definition \ref{def:parabolic bundles} that $L_{\alpha} = L_0$ or $L_\alpha = L_1 = L(-D)$. Thus, the set of weights for a line bundle $L_\bullet$ should be singleton, namely $\{\text{min}(0\leq \alpha<1\mid L_\alpha = L_0)\}$. Now, the collection of weights for the dual line bundle $L^{\dual}_\bullet$ and the tensor product $(L^{\dual}_{\bullet}\otimes L_\bullet)_\bullet$ is known explicitly (cf. \cite[\S 2.1 pp. 310]{Bis03}), from which it directly follows that $0$ is always a parabolic weight of $(L^{\dual}\otimes L_\bullet)_\bullet$, and thus it is the only weight (as the set of weights is singleton). Thus $(L^{\dual}\otimes L_\bullet)_\bullet \simeq \parhomsheaf(L_\bullet\ ,L_\bullet)_\bullet$ has the special parabolic structure \eqref{def:special parabolic structure}, namely $$\parhomsheaf(L_\bullet\ ,L_\bullet)_t = \parhomsheaf(L_\bullet\ ,L_\bullet)([-t]D)\,\forall\ t\in \bb{R}\ .$$
		Thus the morphism $i_{\bullet} : (\mc{O}_X)_\bullet \rightarrow \parhomsheaf(L_\bullet\ ,L_\bullet)_\bullet$ is completely determined by the map $i_0 : \mc{O}_X \rightarrow \parhomsheaf(L_\bullet\ ,L_\bullet) $, which is an isomorphism since $\parhomsheaf(L_\bullet\ ,L_\bullet) \subset \homsheaf(L_0\ ,L_0)\simeq \mc{O}_X$. It follows that $i_\bullet$ is also an isomorphism. 
	\end{remark}
	\begin{definition}\label{def:orthogonal and symplectic parabolic bundles}
		Let $L_{\bullet}$ be a fixed parabolic line bundle on $X$.\hfill
		\begin{enumerate}[(1)]
			\item An \textit{orthogonal parabolic bundle on} $X$  \text{taking values in} $L_{\bullet}$ is a pair $(E_{\bullet}\,,\phi_{\bullet})$, where $E_{\bullet}$ is a parabolic bundle on $X$ together with a parabolic morphism $\phi_{\bullet}:(E_{\bullet}\otimes E_{\bullet})_{\bullet}\rightarrow L_{\bullet}$, such that the map $\widetilde{\phi_{\bullet}}$ in \eqref{eqn:tilde of a parabolic bilinear form} is \textit{symmetric} $\left(\text{i.e.}\ \  (\widetilde{\phi_{\bullet}})^t = \widetilde{\phi_{\bullet}}\right)$ and a \textit{parabolic isomorphism}.
			\item A \textit{symplectic parabolic bundle on} $X$ \textit{taking values in} $L_{\bullet}$ is a pair $(E_{\bullet}\,,\phi_{\bullet})$, where $E_{\bullet}$ is a parabolic bundle on $X$ together with a parabolic morphism $\phi_{\bullet}:(E_{\bullet}\otimes E_{\bullet})_{\bullet}\rightarrow L_{\bullet}$, such that the map $\widetilde{\phi_{\bullet}}$ in \eqref{eqn:tilde of a parabolic bilinear form} is \textit{anti-symmetric} $\left(\text{i.e.}\ \  (\widetilde{\phi_{\bullet}})^t = -\widetilde{\phi_{\bullet}}\right)$  and a \textit{ parabolic isomorphism}.
		\end{enumerate}
		If there is no scope of confusion, we shall simply call them as orthogonal or symplectic parabolic bundles, without mentioning the parabolic line bundle $L_{\bullet}$.
	\end{definition}	
	\section{Biswas-Borne correspondence for orthogonal and symplectic parabolic bundles}
	Consider the stack of $r$-th roots $\mf{X}=\mf{X}_{(\mc{O}(D),s_D,r)}$ together with the natural morphism of stacks $\pi: \mf{X}\rightarrow X$ as in \S \ref{subsection:root stacks}.  Let $\textbf{\text{Vect}}(\mf{X})$ denote the category of algebraic vector bundles on $\mf{X}$. Let $\textbf{\text{PVect}}_{\frac{1}{r}}(X,D)$ denote the category of parabolic bundles over $X$ with parabolic structure along $D$ and weights in $\frac{1}{r}{\bb Z}$. It is clear from Definition \ref{def:parabolic bundles} that parabolic bundles with weights in $\frac{1}{r}{\bb Z}$ can as well be regarded as functors $E_{\bullet} : \left(\frac{1}{r}\bb Z\right)^{op}\rightarrow \textbf{\text{Vect}}(X)$.\\ It has been shown in \cite{Bor07} that the categories $\textbf{\text{PVect}}_{\frac{1}{r}}(X,D)$ and $\text{\textbf{Vect}}(\mf{X})$ are equivalent. Let us recall the correspondence.
	\begin{theorem}\cite[Th\'eor\`eme 2.4.7]{Bor07}\label{thm:Biswas-Borne-correspondence} 
		There is an equivalence of categories between $\Vect(\mf{X})$ and $\textnormal{\textbf{PVect}}_{\frac{1}{r}}(X,D)$ 
		given by the following correspondence:
		\begin{enumerate}[$\bullet$]
			\item A vector bundle $\mathcal{E}$ on $\mf{X}$ gives rise to the 
			parabolic vector bundle $E_{\bullet}$ on $X$
			given by 
			$$\frac{\ell}{r} \longmapsto E_{\frac{\ell}{r}} 
			:= \pi_*\left(\mathcal{E}\otimes_{\mc O_{\mf X}}\ms M^{-\ell}\right),\,\,\forall\,\ell\,\in \bb Z.$$ 
			
			\item Conversely, suppose $\ms{M}$ denote the tautological line bundle on $\mf{X}$ satisfying $\pi^*(\mc{O}(D))\simeq \ms{M}^{\otimes r}$ \cite[\S 3.2]{Bor07}. A parabolic vector bundle $E_{\bullet}$ on $X$ 
			gives rise to the vector bundle $\mathcal{E}$ on $\mf{X}$ given by 
			$$\mathcal{E} := \int^{\frac{\ell}{r}\in\frac{1}{r}\bb{Z}} 
			\pi^*(E_{\frac{\ell}{r}})\otimes_{{\mc O}_{\mf X}} \ms{M}^{\otimes \ell}\ ,$$
			where $\int^{\frac{1}{r}\mathbb{Z}}$ stands for the coend.
		\end{enumerate}
	\end{theorem}
	Both $\ParVect_{\frac{1}{r}}(X,D)$ and $\Vect(\mf{X})$ are additive tensor categories which also admit internal Hom objects. The following lemma shows that the equivalence of categories above preserves tensor products, duals and internal Hom objects as well.
	\begin{lemma}\
		Consider the equivalence between $\Vect(\mf{X})$ and $\ParVect_{\frac{1}{r}}(X,D)$ in Theorem \ref{thm:Biswas-Borne-correspondence}. Let $E_{\bullet}$ and $F_{\bullet}$ be two parabolic bundles on $X$ which correspond to the vector bundles $\mc{E}$ and $\mc{F}$ on $\mf{X}$ respectively. Then the parabolic dual $(E_{\bullet}\otimes F_{\bullet})_{\bullet}$ on $X$ corresponds to the dual $\mc{E}\otimes \mc{F}$ on $\mf{X}$, and the parabolic \textnormal{Hom}-sheaf $\parhomsheaf(E_{\bullet}\,,F_{\bullet})_{\bullet}$ on $X$ corresponds to the \textnormal{Hom}-sheaf $\homsheaf_{\mc{O}_{\mf{X}}}(\mc{E}\,,\mc{F})$ on $\mf{X}$. Consequently, the parabolic dual $E^{\dual}_{\bullet}$ on $X$ corresponds to the dual $\mc{E}^{\dual}$ on $\mf X$.
	\end{lemma} 
	
	\begin{proof}
		The correspondence between tensor products follows from \cite[Th\'eor\`eme 3.13]{Bor09}, while the correspondence for internal Hom-sheaves follows from \cite[Theorem 6.1 and Example 6.2]{BorVis12}\,. Also, note that by definition 
		\begin{align}\label{eqn:duals}
			\mc{E}^{\dual} = \homsheaf(\mc{E}\ ,\mc{O}_{\mf{X}})\ \ \text{and}\ \  E^{\dual}_{\bullet} = \parhomsheaf(E_{\bullet}\ , (\mc{O}_X)_{\bullet})_{\bullet}
		\end{align}
		where $(\mc{O}_X)_{\bullet}$ denote the sheaf $\mc{O}_X$ with special structure \eqref{def:special parabolic structure}. We have just seen that the internal Hom-sheaves are respected under the correspondence. To show that $\mc{E}^{\dual}$ corresponds to $E^{\dual}_{\bullet}$, it is thus enough to show that $\mc{O}_{\mf{X}}$ on $\mf{X}$ corresponds to the parabolic sheaf $(\mc{O}_X)_{\bullet}$ on $X$. 
		To see this, we use the following isomorphism of vector bundles on $\mf{X}$,
		\begin{align}
			&\int^{\frac{\ell}{r}\in\frac{1}{r}\bb{Z}} 
			\pi^*\left(\mc{O}_X\left(\left[-\frac{\ell+m}{r}\right]D\right)\right)\otimes_{{\mc O}_{\mf X}} \ms{M}^{\otimes \ell} \simeq \ms{M}^{\otimes -m}\ \ \forall\ m\in\bb{Z}\qquad \text{\cite[Lemme 3.17]{Bor07}}\\
			\implies&\int^{\frac{\ell}{r}\in\frac{1}{r}\bb{Z}} 
			\pi^*\left(\mc{O}_X\left(\left[-\frac{\ell}{r}\right]D\right)\right)\otimes_{{\mc O}_{\mf X}} \ms{M}^{\otimes \ell} \simeq \mc{O}_{\mf{X}}\qquad(\text{for}\ m=0)\label{eqn:trivial parabolic bundle corresponds to trivial bundle on stack}
		\end{align}
		The term in LHS of the last isomorphism is precisely the vector bundle on $\mf{X}$ corresponding to $(\mc{O}_X)_\bullet$ (cf. statement of Theorem \ref{thm:Biswas-Borne-correspondence}). Thus \eqref{eqn:trivial parabolic bundle corresponds to trivial bundle on stack} precisely says that $(\mc{O}_X)_{\bullet}$ corresponds to $\mc{O}_{\mf{X}}$, as claimed.
	\end{proof} 
	
	\begin{lemma}\label{lem:tilde correspondence}
		Consider a vector bundle $\mc{E}$ on $\mf{X}$ together with a vector bundle morphism $\varphi : \mc{E}\otimes \mc{E} \rightarrow \mc{L}$ taking values in a line bundle $\mc{L}$. Under the equivalence of Theorem \ref{thm:Biswas-Borne-correspondence}, suppose $\mc{E}, \mc{L}$ and $\varphi$ correspond to $E_{\bullet}, L_\bullet$ and $\phi_{\bullet}: (E_\bullet\otimes E_\bullet)_{\bullet}\rightarrow L_\bullet$ on $X$ respectively.  Then the morphism $\widetilde{\varphi}$ of \eqref{eqn:tilde of a stacky bilinear form} corresponds to the parabolic morphism $\widetilde{\phi_{\bullet}}$ of \eqref{eqn:tilde of a parabolic bilinear form} .
	\end{lemma}
	\begin{proof}
		Recall that the morphisms $\widetilde{\varphi}$ and $\widetilde{\phi_{\bullet}}$ are defined as compositions \eqref{eqn:stacky tilde composition} and \eqref{eqn:parabolic tilde composition} respectively. We recall them here for convenience:
		\begin{align}
			&\mc{E} \simeq \mc{E}\otimes \mc{O}_{\mf{X}} \stackrel{j}{\rightarrow} \mc{E}\otimes (\mc{E}^{\dual}\otimes \mc{E}) \simeq \mc{E}^{\dual}\otimes\left(\mc{E}\otimes \mc{E}\right) \stackrel{\text{Id}_{\mc{E}^{\vee}}\otimes\varphi}{\DOTSB\relbar\joinrel\relbar\joinrel\relbar\joinrel\rightarrow} \mc{E}^{\dual}\otimes\mc{L}\\
			&E_{\bullet} \simeq (E_{\bullet}\otimes (\mc{O}_X)_{\bullet})_{\bullet} \stackrel{\iota_{\bullet}}{\longrightarrow}(E_{\bullet}\otimes (E^{\dual}_{\bullet}\otimes E_{\bullet})_{\bullet})_{\bullet}\simeq (E_{\bullet}^{\dual}\otimes (E_{\bullet}\otimes E_{\bullet})_{\bullet})_{\bullet} \stackrel{\text{Id}\otimes \phi_{\bullet}}{\DOTSB\relbar\joinrel\relbar\joinrel\relbar\joinrel\rightarrow} (E_{\bullet}^{\dual}\otimes L_{\bullet})_{\bullet}			
		\end{align}
		Comparing the bundles and morphisms appearing in these two compositions under the equivalence of Theorem \ref{thm:Biswas-Borne-correspondence},  we see immediately that at each stage of the two compositions they correspond to each other, except possibly for $j$ and $\iota_{\bullet}$. Thus it is enough to show that $j$ corresponds to $\iota_{\bullet}$ under the equivalence. To this end, we observe the following general fact for categories, which is easy to check: suppose that $\ms{C}$ and $\ms{D}$ be two categories. Let $F:\ms{C}\rightarrow\ms{D}$ be a functor. Let $\alpha: A\simeq A'$ and $\beta: B\simeq B'$ be two isomorphisms in $\ms{C}$. Then the following diagram is commutative:
		\begin{align}\label{diagram:generality}
			\xymatrix{\text{Hom}_{\ms{C}}(A,B) \ar[rr] \ar[d] && \text{Hom}_{\ms{C}}(A',B') \ar[d] \\
				\text{Hom}_{\ms{D}}(F(A),F(B)) \ar[rr] && \text{Hom}_{\ms{D}}(F(A'),F(B'))
			}
		\end{align}
		where the upper horizontal arrow  sends $f\mapsto \beta\circ f\circ \alpha^{-1}$ (similarly the lower horizontal arrow sends $g\mapsto F(\beta)\circ g\circ F(\alpha)^{-1}$).\\
		In particular, applying this observation for the categories $\ms{C}=\Vect(\mf{X})$ and $\ms{D}=\ParVect_{\frac{1}{r}}(X,D)$ together with the equivalence of Theorem \ref{thm:Biswas-Borne-correspondence} as the functor $F$, and replacing $\alpha$ by the canonical isomorphisms $\text{Id}_{(\mc{O}_X)_{\bullet}} : (\mc{O}_X)_{\bullet} \rightarrow (\mc{O}_X)_{\bullet}$, and
		$\beta$ by the isomorphism $$\parhomsheaf(E_{\bullet}\ ,E_{\bullet})_{\bullet} \simeq \parhomsheaf((\mc{O}_X)_{\bullet}\ ,\parhomsheaf(E_{\bullet}\ ,E_{\bullet})_{\bullet})_{\bullet}\ ,\qquad\text{\cite[Example 3.3]{Yok95}}$$
		we obtain the commutative diagram
		\diagram{
			\text{Hom}(\mc{O}_{\mf{X}}\ ,\homsheaf(\mc{E}\ ,\mc{E})) \ar[rr] \ar[d] && \text{Hom}\left(\mc{O}_{\mf{X}}\ , \homsheaf(\mc{O}_{\mf{X}}\ ,\homsheaf(\mc{E}\ ,\mc{E}))\right) \ar[d]\\
			\text{PHom}((\mc{O}_X)_{\bullet}\ ,\parhomsheaf(E_{\bullet}\ ,E_{\bullet})_{\bullet}) \ar[rr] && \text{PHom}((\mc{O}_X)_{\bullet}\ ,\parhomsheaf((\mc{O}_X)_{\bullet}\ ,\parhomsheaf(E_{\bullet}\ ,E_{\bullet})_{\bullet})\\
		}
		The identity morphism of $\mc{E}$ gives rise to an element in the top left Hom set, which clearly maps to $j$ via the upper horizontal map, and to the map associated to $\text{Id}_{E_{\bullet}}$ via the left vertical map. Since $\iota_{\bullet}$ is precisely the image of the latter element under the bottom horizontal map, it follows that $j$ maps to $\iota_{\bullet}$ via the right vertical map, which proves our claim.
	\end{proof}
	\begin{theorem}\label{thm:borne correspondence for orthogonal and symplectic bundles}
		Let $X, D$ and $r$ be as above. Fix a parabolic line bundle $L_\bullet$ with weights in $\frac{1}{r}\bb Z$ on $X$, and let $\mc{L}$ denote the corresponding line bundle on the $r$-th root stack $\mf X$  under the equivalence of Theorem \ref{thm:Biswas-Borne-correspondence}. There is an equivalence between the category of orthogonal (respectively, symplectic) vector bundles on the root stack $\mf{X}$ taking values in $\mc L$, and the category of orthogonal (respectively, symplectic) parabolic vector bundles on $X$ taking values in $L_\bullet$\ .
	\end{theorem}
	\begin{proof}
		Let $\mc{E}$ be an orthogonal vector bundle on $\mf{X}$ together with a morphism $\varphi: \mc{E}\otimes\mc{E}\rightarrow\mc{L}$. By Theorem \ref{thm:Biswas-Borne-correspondence}, $\varphi$ corresponds to a parabolic morphism $\phi_\bullet:(E_\bullet\otimes E_\bullet)_{\bullet}\rightarrow L_\bullet$. By Lemma \ref{lem:tilde correspondence}, $\widetilde{\varphi}$ corresponds to $\widetilde{\phi_{\bullet}}$ (\eqref{eqn:tilde of a stacky bilinear form} and \eqref{eqn:tilde of a parabolic bilinear form}). To see that the associated transpose maps $(\widetilde{\varphi})^t$ and $(\widetilde{\phi_{\bullet}})^t$ also correspond, we again compare the compositions \eqref{eqn:transpose alternate definition for stacky vector bundles} and \eqref{eqn:transpose alternate description for parabolic bundles} corresponding to the maps $\widetilde{\varphi}^t$ and $(\widetilde{\phi_{\bullet}})^t$ respectively, in which we see that the bundles and morphisms appearing in the compositions correspond to each other at each stage under the composition of Theorem \ref{thm:Biswas-Borne-correspondence}. It immediately follows that $(\mc{E},\varphi)$ is an orthogonal (respectively, symplectic) bundle on $\mf{X}$ if and only if $(E_\bullet,\phi_\bullet)$ is an orthogonal (respectively, symplectic) parabolic bundle on $X$.
	\end{proof}
	\section{Orthogonal and symplectic bundles on the root stack as reductions of structure groups}
	For this section, we assume our base field to be $k=\mathbb{C}$. For a linear algebraic group $G$, one can consider principal $G$-bundles on a Deligne-Mumford stack $\mf{X}$ of finite type, which are constructed analogous to coherent sheaves on $\mf{X}$ \cite[\S 1.2]{BisMajWon12}. Let $X,D$ and $r$ be as in the previous section, and consider the resulting $r$-th root stack $\mf{X}$ as before. Let $\mc{P}$ be a principal $G$-bundle on $\mf{X}$. For any representation $\rho: G\rightarrow \text{GL}(V)$, the associated bundle construction $\mc{P}\times_{\rho} V$ gives rise to a vector bundle on $\mf{X}$. If $\text{Rep}(G)$ denotes the category of finite-dimensional left linear representations of $G$ (or equivalently $G$-modules), $\mc{P}$ induces in this way  a tensor functor from $\text{Rep}(G)$ to $\Vect(\mf{X})$, which we denote by $F_{\mc{P}}$. The functor $F_{\mc{P}}$ enjoys certain abstract properties \cite{Nor76}:
	\begin{enumerate}[(i)]
		\item $F_{\mc{P}}$ is additive;
		\item It is compatible with the direct sum and tensor product operations of $\text{Rep}(G)$ and $\Vect(\mf{X})$;
		\item It takes the $1$-dimensional trivial representation to the trivial line bundle on $\mf{X}$;
		\item It takes an $n$-dimensional representation to a rank $n$-vector bundle on $\mf{X}$.
	\end{enumerate} 
	Now, consider $X, D$ and $r$ as before. Let $\mf{X}$ denote the $r$-th root stack. Using with the equivalence between $\Vect(\mf{X})$ and $\ParVect_{\frac{1}{r}}(X,D)$ [Theorem \ref{thm:Biswas-Borne-correspondence}], $F_{\mc{P}}$ gives rise to a functor from $\text{Rep}(G)$ to  $\ParVect_{\frac{1}{r}}(X,D)$, which also satisfies the properties (i)-(iv) above. We denote the resulting functor by $E_{\mc{P}}$. The following result shows that the map $\mc{P}\mapsto E_{\mc{P}}$ is actually a one-to-one correspondence.
	\begin{proposition}\label{prop:principal bundles correspond to tensor functors}
		Let $G$ be a linear algebraic group. Sending a principal $G$-bundle $\mc{P}$ on the $r$-th root stack $\mf{X}$ to the functor $E_{\mc{P}}$ as above is a one-to-one correspondence between the class of principal $G$-bundles on $\mf{X}$ and the class of tensor functors from $ \text{Rep}(G)$ to $\ParVect_{\frac{1}{r}}(X,D)$ which satisfy the properties (i)-(iv).
	\end{proposition}
	\begin{proof}
		\cite[Proposition 6.1]{BisMajWon12}.
	\end{proof}
	We now focus on the case when $G$ is the orthogonal or symplectic group. Fix a positive integer $n$. We would like to interpret orthogonal bundles (respectively, symplectic bundles) of rank $n$ on the root stack $\mf{X}$ as principal $O(n)$-bundles (respectively, principal $\textnormal{Sp}(n)$-bundles), where in the symplectic case, $n$ is assumed to be even. 
	
	\begin{proposition}\label{prop:orthogonal and symplectic bundles as reductions of structure group}
		Let $X, D$ and $r$ be given as before. The class of orthogonal (respectively, symplectic) bundles of rank $n$ on the stack $\mf{X}$ of $r$-th roots such that the bilinear form takes values in $\mc{O}_{\mf{X}}$ (cf. Definition \ref{def:orthogonal and symplectic bundles on stacks}) are in bijection with the class of principal $O(n)$-bundles (respectively, principal $\textnormal{Sp}(n)$-bundles) on $\mf{X}$, where in the symplectic case, $n$ is assumed to be even.
	\end{proposition}
	\begin{proof}
		Let us prove the statement in the case of orthogonal bundles, the argument for the case of symplectic bundles being similar. Recall that we have already established a one-to-one correspondence between the class of orthogonal bundles of rank $n$ on $\mf{X}$ taking values in $\mc{O}_{\mf{X}}$ and the orthogonal parabolic bundles of rank $n$ on $X$ taking values in $\mc{O}_X$ [Theorem \ref{thm:borne correspondence for orthogonal and symplectic bundles}]. Moreover, by Proposition \ref{prop:principal bundles correspond to tensor functors} there is a one-to-one correspondence between the class of principal $O(n)$-bundles on $\mf{X}$ and the class of tensor functors from $\text{Rep}(O(n))$ to $\ParVect_{\frac{1}{r}}(X,D)$ which satisfy conditions (i)-(iv) above. Thus, it only remains to establish a one-to-one correspondence between the class of orthogonal parabolic bundles of rank $n$ on $X$ with weights in $\frac{1}{r}\bb{Z}$ and the class of tensor functors from $\text{Rep}(O(n))$ to $\ParVect_{\frac{1}{r}}(X,D)$ which satisfy conditions (i)-(iv) above . 
		We follow the approach of \cite{BalaBisNag01}. \\
		Let $(E_{\bullet}\ ,\varphi_{\bullet})$ be an orthogonal parabolic bundle of rank $n$ and weights in $\frac{1}{r}\bb{Z}$ in the sense of Definition \ref{def:orthogonal and symplectic parabolic bundles}. By a theorem of Kawamata \cite[Theorem 17]{Kaw81}, there exists a connected smooth projective variety $Y$ over $\bb{C}$ and a finite Galois morphism
		\begin{align}
			p:Y\rightarrow X
		\end{align} 
		ramified along $D$ with ramification index $r$ and satisfying certain properties. Let $\Gamma$ be the Galois group of $p$. By the correspondence between parabolic bundles on $X$ and $\Gamma$-equivariant bundles on $Y$ \cite{Bis97}, the underlying parabolic vector $E_{\bullet}$ corresponds to a $\Gamma$-equivariant bundle $E'$ on $Y$. Moreover, the correspondence is a tensor equivalence, and thus the bilinear form $\varphi_{\bullet}$ induces a $\Gamma$-equivariant bilinear form $\varphi'$ on $E'$ which takes values in $\mc{O}_Y$. In fact, the pair $(E',\varphi')$ is an orthogonal $\Gamma$-equivariant bundle on $Y$ \cite[\S 2.4]{BisMajWon11}. Since the fibers of $E'$ have $\Gamma$-equivariant orthogonal structures, the frame bundle $Fr(E')$ admits the structure of a $(\Gamma\ ,O(n))$-bundle on $Y$ in the sense of \cite[\S 2]{BalaBisNag01}. Then, given any left linear representation of $O(n)$, the associated bundle construction gives rise to a $\Gamma$-equivariant bundle on $Y$ of rank $n$, which gives rise to a parabolic bundle on $X$ of rank $n$, again via the correspondence of \cite{Bis97}. Thus $Fr(E')$ produces a functor from the category $\text{Rep}(O(n))$ to the category $\ParVect(X,D)$ of parabolic bundles over $X$\,. It is straightforward to check that this functor does not depend on the choice of the morphism $p$ by going to a further covering (cf. the proof of \cite[Proposition 2.6]{BalaBisNag01}). Thus, starting from an orthogonal parabolic bundle $(E_{\bullet}\ ,\varphi_{\bullet})$ of rank $n$ on $X$ we end up producing a functor $F_{(E_{\bullet},\varphi_{\bullet})}$ from $\text{Rep}(O(n))$ to $\ParVect(X,D)$. We would like to say that the image of $F_{(E_{\bullet},\varphi_{\bullet})}$ actually lands in $\ParVect_{\frac{1}{r}}(X,D)$.
		Now,  it is clear that parabolic duals $E_{\bullet}^{\dual}$ and all tensor powers $(\otimes ^k E_{\bullet})_{\bullet}$ also have weights in $\frac{1}{r}\bb{Z}$. Consequently, any parabolic sub-bundle of $(\otimes^k E_{\bullet})_{\bullet}$ with its induced parabolic structure also has weights in $\frac{1}{r}\bb{Z}$. The functor $F_{(E_{\bullet},\varphi_{\bullet})}$ clearly sends the standard representation $\bb{C}^n$ of $O(n)$ to $E_{\bullet}$. Since each irreducible representation of $O(n)$ occurs as a subrepresentation of some $\otimes^k\  \bb{C}^n$ \cite[Chapter 10]{GooWal}, it immediately follows that the image of $F_{(E_{\bullet},\varphi_{\bullet})}$ in fact lands inside $\ParVect_{\frac{1}{r}}(X,D)$.\\
		In the case of symplectic bundles, the argument is analogous. The fact that each irreducible representation of $\textnormal{Sp}(n)$ ($n$ even) occurs as a subrepresentation of some $\otimes^k\  \bb{C}^n$ again follows from  \cite[Chapter 10]{GooWal}.
		%	If we ignore the orthogonal structure for the time being and consider the underlying parabolic bundle $E_{\bullet}$, the same procedure above produces a functor $F_{E_{\bullet}}$ from $\text{Rep}(GL(n))$ to $\ParVect_{\frac{1}{r}}(X,D)$ \cite[Proposition 2.6]{BalaBisNag01}. Note that the inclusion $\iota : O(n)\subset GL(n)$ induces a functor $\iota^*$ from $\text{Rep}(GL(n))$ to $\text{rep}(O(n))$, and moreover, the following equality of functors holds:
		%	$$F_{E_{\bullet}} = F_{(E_{\bullet}\,,\varphi_{\bullet})}\circ \iota^*\,.$$
		%	As the functor $F_{E_{\bullet}}$ maps to $\ParVect_{\frac{1}{r}}(X,D)$, we 
	\end{proof}
	\begin{remark}
		Alternatively, one can prove Proposition \ref{prop:orthogonal and symplectic bundles as reductions of structure group} by considering \'etale atlases for $\mf{X}$. It amounts to proving the statement first on each such atlas , and showing compatibility of the construction under the change of atlas maps \eqref{diagram:smooth atlas}. We have taken a different approach for the proof, as it serves as a different approach to looking at the problem, and also as an application of the results proved so far.
	\end{remark}
	\section{Orthogonal and symplectic parabolic connections}\label{section:orthogonal and symplectic parabolic connections}
	Let us consider a smooth projective variety $X$ over $k$. Let $D$ be a strict normal crossings divisor on $X$, meaning that $D$ is a reduced effective Cartier divisor such that its irreducible components are non-singular and they intersect each other transversally. We shall denote such divisors by 'sncd' for short in what follows. The condition of sncd serves two purposes. First, it allows us to talk about logarithmic sheaf of differentials on $X$ and logarithmic connections \cite{EV}. Second, In order to discuss connections on root stacks, we shall require smoothness assumptions on these stacks, which will follow if $D$ is sncd.\\
	From here onward, we shall need to use a slightly modified definition for parabolic bundles and root stacks, as given in \cite{Bor09} (see also \cite{BorLar23}). Let ${\bf D}:=(D_i)_{1\leq i\leq s}$ be the set of connected components of the sncd $D$. Let us fix an $s$-tuple ${\bf r} = (r_1,r_2,\cdots,r_s)$ of positive integers.
	For each $i$, consider the root stack $\mf{X}_{(\mc{O}(D_i), s_{D_i},r_i)}$ (cf. \S \ref{subsection:root stacks}), where $s_{D_i}$ denotes the canonical section. We shall be interested in their fiber product, namely 
	\begin{align}\label{def:modified root stack}
		\mf{X}_{(\bf{D},\bf r)} : = \prod_{X, 1\leq i \leq s}\mf{X}_{(\mc{O}(D_i), s_{D_i},r_i)}\ .
	\end{align}
	This is a smooth Deligne-Mumford stack of finite type over $k$ \cite[Proposition 2.4]{BorLar23}. As a fibred category, for any $k$-scheme $T$ the objects in $\mf{X}_{(\bf{D},\bf r)}(T)$ are given by the data of a morphism $f:T\rightarrow X$ of $k$-schemes together with generalized Cartier divisors (cf. \S \ref{subsection:root stacks}) $(\mc{L}_i\ ,s_i)$ on $X$ and isomorphisms $\phi_i : \mc{L}_i^{\otimes r_i} \simeq f^*(\mc{O}(D_i))$ satisfying $\phi_i(s_i^{\otimes r_i})=f^*(s_{D_i})$ for all $1\leq i\leq s$.
	
	Consider $\frac{1}{\bf r}\bb{Z}^s = \frac{1}{r_1}\bb{Z}\times \frac{1}{r_2}\bb{Z}\times\cdots\times\frac{1}{r_s}\bb{Z}$ as an indexing category, such that for any two multi-indices ${\bf w} = (w_1,w_2,\cdots,w_s)$ and ${\bf w'}=(w'_1,w'_2,\cdots,w'_s)$, there is a unique arrow $\bf{w}\rightarrow\bf{w'}$ iff $w_i\leq w'_i\ \forall i$. As before, we denote by $\textnormal{\bf Vect}(X)$ the category of algebraic vector bundles on $X$.
	Unless specified otherwise, we shall only consider our divisor $D$ to be sncd for the rest of the discussion.\\
	Let us now see a modified version of our Definition \ref{def:parabolic bundles} of parabolic bundles due to N. Borne \cite{Bor09}.
	\begin{definition}\text{\cite[D\'efinition 2.1.2]{Bor09}}\label{def:modified parabolic bundle}
		Let $D$ be a sncd on $X$ with components ${\bf D} = (D_i)_{1\leq i\leq s}$. A \textit{parabolic vector bundle} on $X$ with weights in $\frac{1}{\bf r}\mathbb{Z}^s$ and parabolic structure along $D$ consists of the following data:
		\begin{enumerate}[$\bullet$]
			\item A functor $E_{\bullet}: \left(\frac{1}{\bf r}\bb{Z}^s\right)^{op} \longrightarrow \textnormal{\bf{Vect}}(X)\ ,$ and 
			\item  For any integral multi-index ${\bf l}\in \mathbb{Z}^s\subset \frac{1}{\bf r}\bb{Z}^s$, a natural isomorphism $$j_{\bullet +\bf l} : E_{\bullet + {\bf l}} \stackrel{\simeq}{\longrightarrow} E_{\bullet}\otimes \mc{O}_X(-{\bf l\cdot D}) := E_{\bullet}\otimes \mc{O}_X(-\sum_{i=1}^{s}l_iD_i)$$
			making the following diagram commute for all ${\bf l\geq 0}:$
			\begin{align}
				\xymatrix{
					E_{\bullet + \bf l} \ar[r] \ar[d]_{\simeq} & E_{\bullet} \\
					E_{\bullet}\otimes \mc{O}_X(-{\bf l\cdot D}) \ar[ru] 
				}
			\end{align}
		\end{enumerate}
	\end{definition}
	A parabolic morphism between to parabolic bundles is defined in the same way as in Definition \ref{def:parabolic bundles}. Let $\textnormal{{\bf ParVect}}_{\frac{1}{\bf r}}(X,{\bf D})$ denote the category of parabolic vector bundles on $X$. 
	
	\begin{remark}
		Let $X,\bf{D,r}$ be chosen as above. Now, Theorem \ref{thm:Biswas-Borne-correspondence} has a natural analogue showing the equivalence between the category of parabolic bundles on $X$ with weights in $\frac{1}{\bf r}\mathbb{Z}^s$ and the category of algebraic vector bundles on $\mf{X}_{(\bf D,r)}$ \cite[Th\`eor\'eme 2.4.7]{Bor09}. It is clear that the notions of orthogonal and symplectic bilinear forms on a parabolic vector bundle as in Definition \ref{def:orthogonal and symplectic parabolic bundles} also make sense in this setting with the appropriate changes. The results obtained in previous sections continue to hold true in this setting as well after the modifications. in particular, the analogue of Theorem \ref{thm:borne correspondence for orthogonal and symplectic bundles} holds true in this setup as well.
		
	\end{remark}
	\subsection{Logarithmic connections and parabolic connections}
	Let $\Omega^1_{X/k}(\log D)$ denote the sheaf of meromorphic differentials on $X$ having logarithmic poles along $D$. It fits into a short-exact sequence  \cite[Proposition 2.3]{EV}:
	\begin{align}
		0\rightarrow \Omega^1_{X/k} \rightarrow \Omega^1_{X/k}(\log D) \xrightarrow{res} \oplus_{i=1}^{s} \mc{O}_{D_i} \rightarrow 0
	\end{align}
	where $res$ denotes the residue map, $D_i$'s are the components of $D$, and $\mc{O}_{D_i}$ denotes the structure sheaf of the closed subscheme $D_i$ of $X$. 
	\begin{definition}\label{def:logarithmic connection}
		Let $E_{\bullet}$ be a vector bundle on $X$. 
		\begin{enumerate}[(i)]
			\item A \textit{logarithmic connection} on $E$ is a $k$-linear sheaf homomorphism
			\begin{align}
				\nabla : E \to E \otimes_{\mc{O}_X} \Omega^1_{X/k}(\log D)
			\end{align}
			satisfying the Leibniz rule: 
			$$\nabla(f\cdot \sigma) = f\nabla \sigma + \sigma\otimes df,$$
			for all locally defined sections $f$ and $\sigma$ of $\mc{O}_X$ and $E$, respectively.  
			It is usually denoted as a pair $(E, \nabla)$. 
			If the image of $\nabla$ lands in the subsheaf $E\otimes_{\mc{O}_X}\Omega^1_{X/k}$\ , 
			$\nabla$ is said to be a {\it algebraic connection} on $E$.
			\item A \textit{morphism of logarithmic connections} between $(E,\nabla)$ and $(E',\nabla')$ is a vector bundle morphism $\phi: E\rightarrow E'$ such that the diagram below commutes:
			\begin{align}\label{morphism of connections diagram}
				\xymatrix{
					E \ar[rr]^(.3){\nabla} \ar[d]_{\phi} && E\otimes \Omega^1_{X/k}(\log D) \ar[d]^{\phi\otimes \textnormal{Id}} \\
					E' \ar[rr]^(.3){\nabla'} && E'\otimes \Omega^1_{X/k}(\log D)
				}
			\end{align}
		\end{enumerate}
	\end{definition}
	Let us denote by $\textnormal{\bf Con}(X,D)$ the category of logarithmic connections on $(X,D)$. \\
	\begin{remark}\label{rem:canonical logarithmic connection}
		\begin{enumerate}[(1)]
			\item Given a Cartier divisor $B = \sum_{i=1}^{n}\mu_iD_i$, one can always construct a canonical logarithmic connection on the line bundle $\mc{O}_X(B)$ as follows \cite[Lemma 2.7]{EV}: let $x_i$ be a local equation of $D_i$. We define the logarithmic connection
			\begin{align}\label{def:canonical logarithmic connection}
				d(B) : \mc{O}_X(B)  \longrightarrow \mc{O}_X(B)\otimes \Omega^1_{X/k}(log D)
			\end{align}
			given by the formula $d(B)\left(\prod_{i=1}^{n}x_i^{-\mu_i}\right)  = -\prod_{i=1}^{n}x_i^{-\mu_i}\cdot \sum_{i=1}^{n} \mu_i\dfrac{dx_i}{x_i}\ .$
			\item Given two logarithmic connections $(E,\nabla)$ and $(E',\nabla')$ with poles along $D$, there is a well-defined notion of tensor product $(E\otimes E',\nabla\otimes \text{Id}_{E'}+ \text{Id}_E\otimes \nabla')$, which is again a logarithmic connection with poles alond $D$ [\textit{loc. cit.}]. 
			\item Using (2), one can define the 'twist' of a logarithmic connection $(E,\nabla)$ by a Cartier divisor $B$ to be the tensor product of $(E,\nabla)$ and $(\mc{O}_X(B), d(B))$, where $d(B)$ is the canonical connection on $\mc{O}_X(B)$ \eqref{def:canonical logarithmic connection}. We shall denote this twisted connection by $(E\otimes \mc{O}_X(B),\nabla(B))$.
		\end{enumerate}
	\end{remark} 
	Let $X, {\bf D, r}$ be chosen as in the beginning of this section. Also, recall from Definition \ref{def:modified parabolic bundle} the multi-index notation ${\bf l\cdot D} := \sum_{i=1}^{s}l_iD_i$.
	\begin{definition}\label{def:parabolic connection}
		A \textit{parabolic connection} on $(X, {\bf D})$ with weights in $\frac{1}{\bf r}\mathbb{Z}$ consists of the following data:
		\begin{enumerate}[$\bullet$]
			\item A functor $(E_\bullet,\nabla_{\bullet}) : \left(\frac{1}{\bf r}\mathbb{Z}^s\right)^{op}\longrightarrow \textnormal{\bf Con}(X,D)$\ , and
			\item  for each integral multi-index ${\bf l}$\ , a natural isomorphism 
			\begin{align}
				j_{\bullet + \bf l} : (E_{\bullet + {\bf l}},\nabla_{\bullet +\bf l}) \simeq (E_\bullet\otimes \mc{O}_X(-{\bf l\cdot D}), \nabla_{\bullet}(-{\bf l\cdot D}))\ \ \ (\textnormal{cf. Remark}\ \ref{rem:canonical logarithmic connection})
			\end{align}
			making the following diagram commute for all $\bf l\geq 0$:
			\begin{align}
				\xymatrix{
					(E_{\bullet + \bf l},\nabla_{\bullet+\bf l}) \ar[dd]^{\simeq}_{j_{\bullet+\bf l}} \ar[r] & (E_{\bullet},\nabla_{\bullet})\\
					&\\
					(E_\bullet\otimes \mc{O}_X(-{\bf l\cdot D}), \nabla_{\bullet}(-{\bf l\cdot D})) \ar[uur] &
				}
			\end{align}
		\end{enumerate}
		where the upper horizontal arrow is induced from the functor $(E_\bullet\ ,\nabla_\bullet)$\ , and the slanted arrow is induced from the canonical inclusion $\mc{O}_X(-\bf{l\cdot D})\hookrightarrow \mc{O}_X$.
	\end{definition}
	\begin{remark}
		Given a parabolic connection $(E_\bullet,\nabla_\bullet)$, it is easy to see that there is an induced parabolic bundle structure $E_\bullet : \left(\frac{1}{\bf r}\mathbb{Z}^s\right)^{op}\rightarrow \textnormal{\bf Vect}(X)$ on the underlying bundle by composing the functor $(E_\bullet, \nabla_\bullet)$ with the forgetful functor $\textnormal{\bf Con}(X,D)\rightarrow \textnormal{\bf Vect}(X)$.
	\end{remark}
	\begin{definition}\text{\cite[Definition 4.3]{BorLar23}}\label{def:strongly parabolic connection}
		A parabolic connection $(E_\bullet\ ,\nabla_\bullet)$  on $(X,\bf{D})$ with weights in $\frac{1}{\bf r}\mathbb{Z}^s$ is said to be \textit{strongly parabolic} if for each $1\leq i\leq s$, the induced residue map
		\begin{align}
			\overline{res_{D_i}(\nabla_{\frac{\bf l}{\bf r}})} : \dfrac{E_{\frac{\bf l}{\bf r}}}{E_{\frac{\bf l+e_i}{\bf r}}}\longrightarrow \dfrac{E_{\frac{\bf l}{\bf r}}}{E_{\frac{\bf l+e_i}{\bf r}}}
		\end{align}
		coincides with $\frac{l_i}{r_i}\cdot\text{Id}$, where $e_i$ is the standard $i$-th basis vector in $\mathbb{Z}^s$.
	\end{definition}
	\subsection{Orthogonal and symplectic parabolic connections}\hfill\\
	Let us fix a parabolic connection $(L_\bullet,\nabla_{L_\bullet})$ on $(X,\bf{D})$ (cf. Definition \ref{def:parabolic connection}), so that the underlying bundle $L_\bullet$ is a parabolic line bundle. Let $(E_\bullet,\nabla_{E_\bullet})$ be another parabolic connection on $(X,\bf{D})$. The parabolic tensor product $(E_\bullet^{\dual}\otimes L_\bullet)_{\bullet}$ then admits the tensor product parabolic connection, namely 
	\begin{align}\label{eqn:tensor product parabolic connection}
		((E^{\dual}_{\bullet}\otimes L_{\bullet})_{\bullet}\ , \nabla_{E^{\dual}_{\bullet}}\otimes \text{Id}_{L_{\bullet}}+\text{Id}_{E^{\dual}_{\bullet}}\otimes\nabla_{L_{\bullet}})\ .
	\end{align}
	Suppose the underlying parabolic bundle $E_\bullet$ also admits an orthogonal or symplectic bilinear form $\phi_{\bullet}: E_\bullet \rightarrow (E_\bullet\otimes L_\bullet)_{\bullet}$  taking values in $L_\bullet$ (cf. Definition \ref{def:orthogonal and symplectic parabolic bundles}). Consider the parabolic morphism $\widetilde{\phi_{\bullet}}: E_\bullet \rightarrow \parhomsheaf(E_\bullet\ ,L_\bullet)_\bullet \simeq (E_\bullet^{\dual}\otimes L_\bullet)_\bullet$ associated to it \eqref{eqn:tilde of a parabolic bilinear form}, which is known to be an isomorphism by definition.
	\begin{definition}\label{def:orthogonal and symplectic parabolic connection}
		We say that a parabolic connection $\nabla_{E_\bullet}: E_\bullet \rightarrow E_\bullet\otimes \Omega^1_{X/k}(\log D)$ is \textit{compatible} with $\phi_\bullet$ if $\widetilde{\phi_\bullet}$ takes the connection $(E_\bullet,\nabla_{E_\bullet})$ to the tensor product parabolic connection in \eqref{eqn:tensor product parabolic connection}, i.e. if the following diagram commutes:
		\begin{align}\label{diagram:orthogonal and symplectic parabolic connections diagram}
			\xymatrix{
				E_\bullet \ar[rrrr]^(.4){\nabla_{E_\bullet}} \ar[d]^{\simeq}_{\widetilde{\phi_{\bullet}}} &&&& E_\bullet\otimes\Omega^1_{X/k}(\log D)\ar[d]_{\simeq}^{\widetilde{\phi_{\bullet}}\otimes \text{Id}} \\
				(E_\bullet^{\dual}\otimes L_\bullet)_\bullet \ar[rrrr]^(.4){\nabla_{E^{\dual}_{\bullet}}\otimes \text{Id}_{L_{\bullet}}+\text{Id}_{E^{\dual}_{\bullet}}\otimes\nabla_{L_{\bullet}}} &&& &	(E_\bullet^{\dual}\otimes L_\bullet)_\bullet\otimes \Omega^1_{X/k}(\log D)
			}
		\end{align} 	
		We call $\nabla_{E_\bullet}$ to be an \textit{orthogonal} (respectively, \textit{symplectic}) \textit{parabolic connection} if the bilinear form $\phi_\bullet$ is orthogonal (respectively, symplectic). 
	\end{definition}
	
	\subsection{Orthogonal and symplectic logarithmic connections on Deligne-Mumford stacks}
	The notions of algebraic and logarithmic connections have their natural analogues for a smooth Deligne-Mumford stack $\mf{X}$ of finite type over $k$. An \textit{effective Cartier divisor} on $ \mf{X}$ is a closed substack $\mf{D}\subset \mf{X}$ such that its ideal sheaf $\mc{I}_{\mf D}$ is invertible. Equivalently, for any \'etale atlas $U\rightarrow \mf{X}$, its pullback is an effective Cartier divisor on $U$. $\mf{D}$ is said to be a \textit{strict normal crossing divisor} (sncd for short) if the same is true for its pullback on an \'etale chart (cf. \cite[pp. 8]{BorLar23}). 
	
	We can construct the sheaf of differentials $\Omega^1_{\mf{X}/k}$ by taking the sheaf of differentials $\Omega^1_{U/k}$ for each \'etale atlas $U\rightarrow \mf{X}$ (cf. \S\ref{subsection:coherent sheaves on stacks}). More generally, for a sncd $\mf{D}$ on $\mf{X}$ one can obtain a stacky version of the sheaf of logarithmic differentials $\Omega^1_{\mf{X}/k}(log\mf{D})$. By a \textit{logarithmic connection} on $(\mf{X},\mf{D})$, we mean a vector bundle $\mc{E}$ together with a $k$-linear morphism $$\nabla_{\mc{E}}: \mc{E}\rightarrow \mc{E}\otimes \Omega^1_{\mf{X}/k}(\log \mf{D})$$
	such that it gives rise to a logarithmic connection on \'etale atlas upon pull-back. $\nabla_{\mc E}$ is said to be \textit{algebraic} if its image lies in $\Omega^1_{\mf{X}/k}$. In particular, one can construct a canonical logarithmic connection $d(\mf{D})$ on the line bundle $\mc{O}_{\mf{X}}(\mf{D})$. (cf. \cite{BorLar23} for more details).
	
	Fix a sncd $\mf{D}$, and a logarithmic connection $(\mc{L},\nabla_{\mc{L}})$ on $(\mf{X},\mf{D})$ so that $\mc{L}$ is a line bundle. Let $(\mc{E}\ ,\nabla_{\mc{E}})$ be another logarithmic connection on $(\mf{X},\mf{D})$. We naturally have a tensor product logarithmic connection 
	$$(\mc{E}\otimes \mc{L}\ , \nabla_{\mc{E}}\otimes \text{Id}_{\mc{L}} + \text{Id}_{\mc{E}}\otimes\nabla_{\mc{L}})$$
	on $(\mf{X},\mf{D})$. Suppose the underlying bundle $\mc{E}$ also admits an orthogonal or symplectic bilinear form $\varphi : \mc{E}\rightarrow \mc{E}\otimes \mc{L}$ taking values in $\mc{L}$ (cf. Definition \ref{def:orthogonal and symplectic bundles on stacks}). Let us Consider the morphism $\widetilde{\varphi}: \mc{E} \rightarrow \homsheaf(\mc{E},\mc{L})\simeq \mc{E}^{\dual}\otimes\mc{L}$ associated to it \eqref{eqn:tilde of a stacky bilinear form}, which is an isomorphism by definition. Imitating what we did in Definition \ref{def:orthogonal and symplectic parabolic connection}, let us call the logarithmic connection $\nabla_{\mc{E}}$ to be \textit{compatible} with $\varphi$ if the following diagram commutes:
	\begin{align}\label{diagram:orthogonal and symplectic stacky logarithmic connections diagram}
		\xymatrix{
			\mc{E} \ar[rrrr]^(.4){\nabla_{\mc E}} \ar[d]^{\simeq}_{\widetilde{\varphi}} &&&& \mc{E}\otimes\Omega^1_{\mf{X}/k}(\log \mf{D})\ar[d]_{\simeq}^{\widetilde{\varphi}\otimes \text{Id}} \\
			\mc{E}^{\dual}\otimes \mc{L} \ar[rrrr]^(.4){\nabla_{\mc{E}^{\dual}}\otimes \text{Id}_{\mc{L}}+\text{Id}_{\mc{E}^{\dual}}\otimes\nabla_{\mc{L}}} &&& &	\mc{E}^{\dual}\otimes \mc{L}\otimes \Omega^1_{\mc{X}/k}(\log \mf{D})
		}
	\end{align}
	We call $\nabla_{\mc E}$ to be an \textit{orthogonal} (respectively, \textit{symplectic}) \textit{logarithmic connection} if the bilinear form $\varphi$ is orthogonal (respectively, symplectic). 	
	
	\subsection{Biswas-Borne correspondence for orthogonal and symplectic parabolic connections}\hfill\\
	We Start with a smooth projective variety $X$ over $k$ together with a sncd $D$ with components $\textbf{D} = (D_i)_{1\leq i\leq s}$ and a multi-index $\textbf{r} = (r_1,\cdots,r_s)\in \bb{Z}_{>0}^s$ as before.   Consider the stack $\mf{X}_{(\bf{D,r})}$ \eqref{def:modified root stack} together with the natural projection $\pi : \mf{X}_{({\bf D,r})}\rightarrow X$. For each $1\leq i\leq s$ there exists an effective Cartier divisor $\mf{D}_i$ on $\mf{X}_{(\bf{D,r})}$ satisfying $\pi^*(D_i) = r_i\mf{D}_i$\ . Moreover, the Cartier divisor $\mf{D} = \sum_{i=1}^{s}r_i\mf{D}_i$ is a sncd on $\mf{X}_{(\bf{D,r})}$ \cite[Remark 2.7]{BorLar23}.
	\begin{lemma}
		The induced morphism $\pi^* \Omega^1_{X/k}(\log D)\rightarrow \Omega^1_{\mf{X_{(\bf D,r)}}/k}(\log\mf{D})$ under the log-\'etale morphism $\pi$ is an isomorphism.
	\end{lemma}
	\begin{proof}
		This is a stacky version of \cite[Lemma 3.5]{BorLar23}. One can prove the statement in the same way by considering \'etale atlases for $\mf{X}_{(\bf{D,r})}$, using the fact that as a log-pair, the map $(\pi,{\bf r}) : (\mf{X}_{(\bf{D,r})}\ ,(\mf{D}_i))\rightarrow (X\ ,(D_i))$ is tautologically log-\'etale \cite[\S3.3]{BorLar23}. 
	\end{proof}
	It is shown in \cite{BorLar23} that there is a tensor equivalence of category of logarithmic connections on $(\mf{X}_{(\bf D,r)},\bf{\mf{D}})$ and the category of parabolic connections on $(X,\bf{D})$ with weights in $\frac{1}{\bf r}\mathbb{Z}^s$. Furthermore, this restricts to give a tensor equivalence between the category of algebraic connections on $(\mf{X}_{(\bf D,r)},\mf{D})$ and the category of strongly parabolic connections on $(X,\bf{D})$ (Definition \ref{def:strongly parabolic connection}). Let us recall the equivalence:
	\begin{enumerate}[(1)]
		\item \cite[Definition 4.18]{BorLar23} To each logarithmic connection $(\mc E, \nabla_{\mc E})$ on $(\mf{X}_{(\bf D,r)}, \mf D)$, one associates a parabolic connection $(E_{\bullet},\widehat{\nabla}_{\bullet})$ on $(X,D)$, where 
		\begin{enumerate}[$\bullet$]
			\item $E_{\bullet}$ is the parabolic bundle given as in Theorem \ref{thm:Biswas-Borne-correspondence}, and 
			\item $\widehat{\nabla}_{\bullet}$ is given by $$\widehat{\nabla}_{\frac{\bf l}{\bf r}} = \pi_*(\nabla_{\mc{E}}(-\bf{l}\cdot\mf D))\,,$$
			here $\nabla(-\bf{l}\cdot\mf D)$ is the tensor product of the connections $(\mc{E},\nabla_{\mc{E}})$ and $(\mc{O}_{\mf{X}_{(\bf D,r)}}(\mf{D}),d(-\bf{l}\cdot\mf D))$\,.\\
			(Note that the push-forward connection makes sense because the morphism of log-pairs $(\pi,r): (\mf X, \mf D)\rightarrow (X,D)$ is log-\'etale, see \cite[Remark 3.8 and Remark 4.19]{BorLar23}).
		\end{enumerate}
		\item \cite[Definition 4.29]{BorLar23} On the other hand, for any such parabolic connection $(E_{\bullet},\widehat{\nabla}_{\bullet})$ on $(X,D)$, one associates the logarithmic connection $(\mc E, \nabla_{\mc E})$ on $\mf X$, where $$\mc E = \int^{\frac{1}{r}\bb{Z}} \pi^*(E_{\bullet})\otimes_ {\mc{O}_{\mf{X}}}{{\mc O}(\bullet r\mf D)}$$ 
		stands for the coend as in Theorem \ref{thm:Biswas-Borne-correspondence}, and $\nabla$ is the unique connection on the coend $\mc E$ compatible with the given connections at each term of the coend (the existence of such a connection is guaranteed by \cite[Lemma 4.27]{BorLar23}).
	\end{enumerate}
	The following theorem is a generalization of the above correspondence for orthogonal or symplectic parabolic connections (cf. Definition \ref{def:orthogonal and symplectic parabolic connection}). 
	\begin{theorem}\label{thm:borne correspondence for orthogonal and symplectic connections}
		Fix a line bundle $\mc{L}$ together with a logarithmic connection $\nabla_{\mc{L}}$ on $(\mf{X}_{(\bf D,r)},\mf{D})$. Let $(L_{\bullet},\nabla_{L_\bullet})$ be the corresponding parabolic connection on $(X,\bf{D})$ under Borne's equivalence. There is a one-to-one correspondence between the class of orthogonal (respectively, symplectic) vector bundles on $\mf{X}_{(\bf D,r)}$ taking values in $\mc{L}$ (Definition \ref{def:orthogonal and symplectic bundles on stacks}) equipped with a compatible logarithmic connection (in the sense of \eqref{diagram:orthogonal and symplectic stacky logarithmic connections diagram}), and the class of orthogonal (respectively, symplectic) parabolic vector bundles on $X$ taking values in $L_\bullet$ (Definition \ref{def:orthogonal and symplectic parabolic bundles}) equipped with a compatible parabolic connection (in the sense of Definition \ref{def:orthogonal and symplectic parabolic connection}).\\
		Furthermore, this equivalence restricts to a correspondence between the class of algebraic connections on $(\mf{X}_{(\bf D,r)},\mf{D})$ compatible with an orthogonal (respectively, symplectic) form and the class of compatible strongly parabolic connections on $(X,\bf D)$ compatible with an orthogonal (respectively, symplectic) form.
		
	\end{theorem}
	\begin{proof}
		Let $(\mc{E},\nabla_{\mc{E}})$ denote a logarithmic connection on $(\mf{X}_{(\bf D,r)},\mf{D})$ which corresponds to the parabolic connection $(E_\bullet,\nabla_{E_\bullet})$ on $(X,\bf{D})$ under the equivalence of \cite{BorLar23}. Consider an orthogonal (respectively, symplectic) form $\varphi$ on the underlying bundle $\mc{E}$ taking values in $\mc{L}$. Now, the associated isomorphism $\widetilde{\varphi}$ of \eqref{eqn:tilde of a stacky bilinear form} gives rise to two logarithmic connections on $\mc{E}^{\dual}\otimes \mc{L}$, namely 
		\begin{align}\label{eqn:induced parabolic connections on stacky tensor product}
			(\widetilde{\varphi}\otimes \text{Id})\circ\nabla_{\mc{E}}\circ (\widetilde{\varphi})^{-1} 
			\ \ \text{and}\ \ \nabla_{\mc{E}^{\dual}}\otimes\text{Id}_{\mc{L}}+\text{Id}_{\mc{E}^{\dual}}\otimes \nabla_{\mc{L}}\ .
		\end{align}
		% 	 on $\mc{E}$ over $(\mf{X}_{(\bf D,r)},\bf{\mf{D}})$ 
		By Theorem \ref{thm:borne correspondence for orthogonal and symplectic bundles}, we obtain an orthogonal (respectively, symplectic) bilinear form $\phi_{\bullet}$ on the underlying bundle $E_\bullet$ corresponding to $\varphi$ on $\mc{E}$. $\phi_\bullet$ takes values in $L_\bullet$ by construction. By Lemma \ref{lem:tilde correspondence} we immediately see that the two parabolic connections mentioned in \eqref{eqn:induced parabolic connections on stacky tensor product} precisely correspond to the parabolic connections
		\begin{align}\label{eqn:induced parabolic connections on parabolic tensor product}
			(\widetilde{\phi_{\bullet}}\otimes \text{Id})\circ\nabla_{E_\bullet}\circ (\widetilde{\phi_{\bullet}})^{-1} 
			\ \ \text{and}\ \ \nabla_{E^{\dual}_\bullet}\otimes\text{Id}_{L_\bullet}+\text{Id}_{E^{\dual}_\bullet}\otimes \nabla_{L_\bullet}\ .
		\end{align}
		on the parabolic dual $(E^{\dual}_\bullet\otimes L_{\bullet})_{\bullet}$ over $(X,\bf D)$, respectively.
		% \begin{align}\label{eqn:induced parabolic connections on stacky tensor product}
			% 	(\widetilde{\varphi}\otimes \text{Id})\circ\nabla_{\mc{E}}\circ (\widetilde{\varphi})^{-1} 
			% 	\ \ \text{and}\ \ \nabla_{\mc{E}^{\dual}}\otimes\text{Id}_{\mc{L}}+\text{Id}_{\mc{E}^{\dual}}\otimes \nabla_{\mc{L}}\ 
			% \end{align}
		% 	on $\mc{E}$ over $(\mf{X}_{(\bf D,r)},\bf{\mf{D}})$, respectively. 
		Now, by our assumption on the compatibility of $\nabla_{\mc{E}}$ with $\varphi$, we know that the two parabolic connections in \eqref{eqn:induced parabolic connections on stacky tensor product} are equal. It immediately follows that the two logarithmic connections in \eqref{eqn:induced parabolic connections on parabolic tensor product} must be equal as well, hence proving our claim. 
		
		The second assertion also follows from the exact same argument by simply restricting our attention to the class of algebraic connections on the stack. 		
		% 	Suppose that $\mc{L}$ be the line bundle on $\mf{X} := \mf{X}_{(\bf r,D)}$ that corresponds to $L_\bullet$ on $X$. let $(\mc{E},\nabla_{\mc{E}})$ be the logarithmic connection on $(\mf{X},\mf{D})$ that corresponds to the parabolic connection $(E_\bullet, \nabla_{E_\bullet}$ on $(X,D)$. The orthogonal or symplectic form $(E_\bullet,\phi_\bullet)$ induces two functors on $\textnormal{\bf ParCon}_{\frac{1}{\bf r}}(X,\bf{D})$ given by \begin{align}
			% 		(E_\bullet, \nabla_{E_\bullet}) &\mapsto \left((E_\bullet\otimes L_{\bullet})_{\bullet}\ ,(\widetilde{\phi_{\bullet}}\otimes \text{Id})\circ\nabla_{E_\bullet}\circ (\widetilde{\phi_{\bullet}})^{-1}\right)\\ 
			% 		\text{and}\ \  (E_\bullet, \nabla_{E_\bullet}) &\mapsto \left((E_\bullet\otimes L_\bullet)_{\bullet}\ ,\nabla_{E_\bullet}\otimes\text{Id}_{L_\bullet}+\text{Id}_{E_\bullet}\otimes \nabla_{L_\bullet}\right)
			% 	\end{align}
		% 	By our assumption, these two functors are equal. 
	\end{proof}
	
\end{document}